\documentclass[12pt]{amsart}
\usepackage{comment}
\usepackage{macros}
\usepackage{enumitem}
\usepackage{stmaryrd}
\usepackage{braket}

\newtheorem*{theorem*}{Theorem}

\usepackage{mathtools}

\newcommand{\ra}{\rightarrow}

\begin{document}
\title[]{Towards Lang--Vojta via degeneration}

\author[Ryan C. Chen]{Ryan C. Chen}
\address{Department of Mathematics,
	Fine Hall - Washington Road.
        Princeton, NJ 08544, USA}
\email[R. C. Chen]{rcchen.math@gmail.com}

\author{Natalia Garcia-Fritz}
\address{ Departamento de Matem\'aticas,
Pontificia Universidad Cat\'olica de Chile.
Facultad de Matem\'aticas,
4860 Av.\ Vicu\~na Mackenna,
Macul, RM, Chile}
\email[N. Garcia-Fritz]{natalia.garcia@uc.cl}%

\author{Siddharth Mathur}
\address{Department of Mathematics, University of Georgia, Athens, GA 30602, USA}
\email[S. Mathur]{siddharth.mathur@uga.edu}%

\author{Hector Pasten}
\address{ Departamento de Matem\'aticas,
Pontificia Universidad Cat\'olica de Chile.
Facultad de Matem\'aticas,
4860 Av.\ Vicu\~na Mackenna,
Macul, RM, Chile}
\email[H. Pasten]{hector.pasten@uc.cl}%

\date{\today}
\subjclass[2020]{Primary: 14G05; Secondary: 14G40, 11G50} %
\keywords{Lang--Vojta conjecture, integral points, moduli stacks, binary forms}%

\begin{abstract} Towards the Lang--Vojta conjecture, we prove results on finiteness and Zariski degeneracy of $S$-integral points of varieties over number fields $k$, including many cases with geometrically irreducible boundary divisors. Our approach builds on the study of arithmetic and geometric properties of moduli spaces of curves with extra structure. 
As an application, we provide families of explicit examples of geometrically irreducible divisors on the projective plane (such as the dual of any smooth curve of degree at least $3$), with respect to which the sets of $S$-integral points  are finite.
Answering a question of  Achenjang and Morrow, we show that, other than the case of curves, every normal projective variety admits a geometrically irreducible divisor $D$ for which finiteness of $(D,S)$-integral points holds over every finite extension of $k$. 
\end{abstract}

\maketitle



\section{Introduction} 

\subsection{The Lang--Vojta conjecture} 

A general conjecture of Lang \cite[Conjecture 5.7]{MR828820} (first asked by Bombieri in the case of surfaces \cite{Noguchi82}) predicts that on a smooth projective variety $X$ of general type over a number field the set of rational points is not dense in the Zariski topology; this conjecture is now known as the Bombieri--Lang conjecture.
Vojta's conjectures \cite[Conjecture 4.3]{MR861984}
extend the Bombieri--Lang conjecture
to pairs $(X,D)$ of log general type and Zariski non-density
of $(D,S)$-integral points.

\begin{conjecture}[The Lang--Vojta conjecture]\label{conjecture:Lang--Vojta} Let $k$ be a number field and $X/k$ a smooth projective variety, with canonical sheaf $\mathcal{K}$. Let $D$ be an effective, reduced, normal crossings divisor on $X$ and let $S$ be any finite set of primes in $\mathcal{O}_k$. If $\mathcal{K}\otimes\mathcal{O}_X(D)$ is big, then any set of $(D,S)$-integral points in $X$ is Zariski degenerate (i.e. not  dense in the Zariski topology).
\end{conjecture}

Roughly, $(D,S)$-integral points are $\mathcal{O}_{k,S}$-points of a model of $X \backslash D$, and we will sometimes refer to these simply as integral points. 
See also \ref{SecIntegralPoints} for a review of these notions, and  \cite{MR3477540} for a detailed discussion on the Lang--Vojta conjecture. 

These problems are central in Diophantine geometry and, for the most part, they remain open, even when $X=\mathbb{P}^2$ and $D$ is a smooth plane curve. The Bombieri--Lang conjecture is known for 
projective varieties with ample and globally generated canonical sheaf, by work of Moriwaki and Noguchi \cite{Moriwaki95,Noguchi82} (see also \cite{MR4381926}), and for subvarieties of abelian varieties, by work of Faltings \cite{MR1109353, MR1307396, MR1109352}. The conjecture is also known for projective varieties admitting non-constant maps to either the moduli stack of smooth curves of genus $g \geq 2$, or to the moduli stack of principally polarized abelian varieties
of any dimension,
also by work of Faltings \cite{MR718935}.
The latter is part of a frequently occurring phenomenon which we exploit in this article: many moduli spaces admit only finitely many integral points. See, for example, \cite{MR3673652, MR4699876, MR4502602,KraemerMaculan2025, lawsawin, MR4132959} for some recent results along these lines.

On the other hand, more is known about the Lang--Vojta conjecture when $D$ has several components, though not exactly under positivity assumptions. This is thanks to a method introduced by Corvaja and Zannier using the Schmidt Subspace Theorem, see \cite{MR2123936, MR1891001}. Important applications of this method have been obtained and we refer, for instance, to \cite{MR2552103}.

\subsection{Arithmetic results}

For the remainder of the introduction, we will work over a fixed number field $k$ and let $S$ be any finite set of primes of $\mathcal{O}_k$. In this article, we study geometrically \emph{irreducible}
and often-singular boundary divisors
$D$ giving Zariski degenerate or even finite sets
of $S$-integral points. In the finite case, we say that
 $(X,D)$ is \emph{arithmetically hyperbolic}, (see Definition \ref{def:zardeg}
and Definition \ref{def:arithhyp} for precise definitions).
Note that we allow $D$ which are not necessarily normal crossings.

We construct many such $D$ as dual varieties of curves and we briefly recall the constructions of such varieties. Let $C \subset \mathbb{P}^n$ be a geometrically integral curve of degree at least $2$,
let $\check{\mathbb{P}}^n$ be the dual projective space, and let
\[
I_C \subset \mathbb{P}^n \times \check{\mathbb{P}}^n
\]
be the closure of the locus of pairs $$(p,[H])$$ where, $p$ is a smooth point of $C$, $H \subset \mathbb{P}^n$ is a hyperplane containing $p$, and $T_p(C) \subset T_p(H)$. The image of $I_C$ under the second projection is the \emph{dual variety} $C^* \subset \check{\mathbb{P}}^n \cong \mathbb{P}^n$ of $C$.
Since the projection $I_C \to C$ is a $\mathbb{P}^{n-2}$-bundle over the smooth locus of $C$, the dual $C^*$ is geometrically irreducible.
These pairs $(X,D) = (\check{\mathbb{P}}^n, C^*)$ 
will be one basic source of examples in this article.

We say that a geometrically integral projective curve $C$
is \emph{smoothly branched} if the normalization of $C$
is unramified over $C$ (Definition \ref{def:regim}).
This is satisfied if, for example, $C$ is at-worst nodal.

\begin{Theorem}[see Theorem \ref{thm:deg4}] \label{thmgenusgreaterone}
Fix a number field $k$ and let $C \subset \mathbb{P}^n$ be a geometrically integral smoothly branched projective curve of geometric genus $g \geq 1$, defined over $k$, which is not contained in any hyperplane. Then
the pair $(\mathbb{P}^n, C^*)$
is arithmetically hyperbolic.
\end{Theorem}

The situation is markedly different for rational curves. Indeed, for a rational curve $C$ of any degree $d \geq 2$ in $\mathbb{P}^2$, the complement of its dual may have infinitely many integral points (see Example~\ref{ex:sharpness}). For $d \geq 4$, however,
we show that finiteness continues to hold after prohibiting high
order tangents to $C$, and that the integral points are never Zariski dense
in general.

\begin{Theorem}[see Proposition \ref{lemma:dualrational}] \label{thmrationalnodal} Fix a number field $k$, let $C \subset \mathbb{P}^n$ be a geometrically integral smoothly branched rational curve of degree $d \geq 4$, defined over $k$, which is not contained in any hyperplane. For any finite set of primes $S$ of $\mathcal{O}_k$, any set of $(C^*,S)$-integral points in $X$ is Zariski degenerate.
Moreover, if no hyperplane section of $C$ pulls back to the normalization
as a divisor with a non-reduced point of multiplicity $\geq d - 1$,
then $(\mathbb{P}^n, C^*)$ is arithmetically hyperbolic.
\end{Theorem}

Equations for dual varieties can be computed explicitly, see, e.g. \cite{zbMATH01886745} for the case of plane curves. Thus, Theorems \ref{thmgenusgreaterone} and \ref{thmrationalnodal} yield families of explicit, geometrically irreducible divisors $D \subset \mathbb{P}^n$ such that any set of $(D,S)$-integral points in $\mathbb{P}^n$ is Zariski degenerate, and many examples with only finitely many $(D,S)$-integral points.
We give a concrete example.

\begin{Example} \label{ex:fermatcubicdual} The dual of the Fermat plane cubic
\[
C:\; X^3+Y^3+Z^3=0 \subset \mathbb P^2
\]
is the sextic $C^* \subset \check{\mathbb P}^2$ with homogeneous dual coordinates $[U:V:W]$ given by the equation
\[
U^{6}+V^{6}+W^{6}-2\bigl(U^{3}V^{3}+U^{3}W^{3}+V^{3}W^{3}\bigr)=0.
\]
This is geometrically irreducible and by Theorem \ref{thmgenusgreaterone},
$(\mathbb{P}^2, C^*)$ is arithmetically hyperbolic.
\end{Example}

Our next result gives a definitive positive answer to a question of Achenjang and Morrow \cite[Question 1]{MR4749147}, who ask: Which smooth projective varieties $X$ admit geometrically irreducible divisors $D$ whose complements admit only Zariski non-dense or finitely many integral points?
The case of $\dim X = 1$ is fully understood: it 
is necessary and sufficient for the genus of $X$ to be at least $1$,
by Siegel's work.
For $\dim X \geq 2$, we show that every $X$
admits a geometrically irreducible divisor $D$ for which $(X,D)$ is arithmetically hyperbolic.

\begin{Theorem}[see Corollary \ref{cor:morrowachen}]\label{thmam}
Let $X$ be any geometrically normal, projective $k$-variety
with $\dim X\ge 2 $. Then there exists a geometrically irreducible divisor $D \subset X$ such that the pair $(X,D)$ is arithmetically hyperbolic.
\end{Theorem}

\subsection{Degenerating classifying maps} To illustrate our method, we present a simplified version of our strategy and demonstrate it in an easy example.

\begin{Strategy} \label{strategy} The following approach propagates arithmetic finiteness using moduli spaces.
\begin{enumerate}
\item Choose a moduli space $\mathcal M$ with finitely many integral points, admitting a geometrically meaningful compactification $\bar{\mathcal M}$ whose boundary $\bar{\mathcal M}\setminus\mathcal M$ parametrizes suitable degenerations. In this paper, we are primarily
interested in variants of moduli of marked curves.

\item Construct a morphism $\phi \colon X \to \bar{\mathcal{M}}$ such that
\[
\phi^{-1}(\bar{\mathcal{M}}\setminus\mathcal{M}) = D,
\]
i.e. produce a family of objects of $\mathcal{S}$ over $X$ degenerating precisely along $D$.

\item Analyze the fibers of $\phi|_{X\setminus D}$ to deduce non-density or finiteness of integral points on $X\setminus D$ (for instance, if $\phi|_{X\setminus D}$ is non-constant, integral points cannot be dense).
\end{enumerate}
\end{Strategy}

\begin{Example} \label{ex:introsample} Using Strategy \ref{strategy}, we outline the proof of Theorem \ref{thmgenusgreaterone} when $g \geq 2$, $n=2$ and $C$ is smooth. \begin{enumerate}

 \item If $C \subset \mathbb{P}^2$ is a smooth projective plane curve of degree $d \geq 4$, and if $\mathcal{M}$ is the Hilbert scheme of \'etale subschemes of degree $d$ in $C$, then $\mathcal{M}$ has only finitely many integral points (see, e.g., \ref{rem:hilbcurveAH}). We compactify this $\mathcal{M}$ by using the Hilbert scheme of \emph{all} degree $d$ subschemes of $C$, call it $\bar{\mathcal{M}}$. 
 
 \item Consider the natural map $\check{\mathbb{P}}^2 \to \bar{\mathcal{M}}$ which sends $[L] \mapsto [L \cap C]$. By definition, $\phi$ degenerates precisely over $C^*$. 
 
 \item The map $\phi|_{\mathbb{P}^2 \setminus C^*}$ is quasi-finite (in fact, injective) since two different lines cannot share a length $d \geq 4$ subscheme. Thus, the finiteness on $\mathcal{M}$ can be transported to $\mathbb{P}^2 \setminus C^*$. \qed
 \end{enumerate}

 \end{Example}

The third step of our strategy is key: to ascend arithmetic properties along $\phi|_{X \setminus D} \colon X \setminus D \to \mathcal{M}$ one at least needs $\phi|_{X \setminus D}$ to be non-constant. Of course, since its image touches the boundary $\phi$ will never be constant, so this seems easy. However, moduli spaces $\bar{\mathcal{M}}$ are often highly non-separated and so $\phi|_{X \setminus D}$ could still be constant even if $\phi$ isn't. The underlying reason is that the objects in the moduli problem may have too much symmetry, as in the following example.

\begin{Example} \label{ex:constant} Consider the moduli stack of $r$-marked genus zero curves \emph{up to linear isomorphism} (see also \ref{defhypersurfaces}), $\mathcal{C}_{r,1}=[\mathbb{P}^r/\mathrm{PGL}_2]$. The family parametrized by $\mathbb{P}^1_{[\lambda, \mu]}$ given by $V(\lambda x^r+\mu y^r) \subset \mathbb{P}^1_{[x,y]} \times \mathbb{P}^1_{[\lambda, \mu]}$ yields a nonconstant map $\mathbb{P}^1_{[\lambda, \mu]} \to \mathcal{C}_{r,1}$ but $\phi|_{\mathbb{P}^1_{[\lambda, \mu]}\setminus \{0,\infty\}}$ is constant. \end{Example}

For the $g \geq 2$ case of Theorem \ref{thmgenusgreaterone},
the
Hilbert scheme of closed subschemes of
a genus $g$ curve $C$
is certainly a separated moduli space.
However, we also need the fact that the interior parameterizing \'etale subschemes
is arithmetically hyperbolic (see Remark \ref{rem:hilbcurveAH}).
The analogous interior is never arithmetically hyperbolic
for genus $g \leq 1$ due to the infinite automorphism
group of the curve $C$. Indeed, these symmetries yield 
infinitely many integral points on the Hilbert scheme,
starting from a single integral point
(Remark \ref{rem:hilbcurve_low_genus_not_AH}).
This phenomenon can be removed after dividing out
by these infinite automorphism groups, at the cost of introducing
non-separatedness toward the boundary
as in Example \ref{ex:constant}. Then we establish arithmetic hyperbolicity for the resulting moduli stacks using the work of Javanpeykar-Loughran \cite{zbMATH07360996} (see Corollary \ref{cor:binaryformsAH}). In this way, we obtain the following result.

\begin{Theorem}[see Theorem \ref{ThmMainCriterion}] \label{ThmMainCriterion:intro}
Let $k$ be a number field, and suppose $X$ is a geometrically integral separated scheme of finite type over $k$. Let $\pi\colon P\to X$ be a proper and smooth morphism of schemes whose geometric fibers are smooth projective connected curves of a fixed genus $g$.
Fix an integer $r \geq 2$
if $g \geq 1$,
and an integer $r \geq 4$ if $g = 0$.

Let $I \subseteq P$ and $\Delta\subsetneq X$ be closed subschemes such that above $X \setminus \Delta$, the map $\pi|_I$ is finite \'etale of degree $r$. Consider the following properties:
\begin{itemize}
\item[(i)] There exists a geometric point $x\in \Delta(\bar{k})$ above which $I$ is finite and flat, and the scheme-theoretic fiber $I_x \to \Spec \bar{k}$ is non-reduced. If $g = 0$, we assume all points in $I_x$ have multiplicity $\leq r-2$.
\item[(ii)] The map $\pi|_I$ is finite flat and for every geometric point $x\in \Delta(\bar{k})$, the scheme-theoretic fiber $I_x \to \Spec \bar{k}$ is non-reduced. If $g = 0$, we assume all points in $I_x$ have multiplicity $\leq r-2$. Suppose also that $\Delta$ meets all closed curves in $X$ (e.g. if $X$ is proper and $\Delta$ is given by an ample divisor). \end{itemize}
If (i) (resp. (ii)) holds then every set of $(\Delta,S)$-integral points
on $X$ is Zariski degenerate (resp. finite).
\end{Theorem}

We consider two methods for controlling
the boundary degeneration.
First, to obtain a version of Theorem \ref{thmrationalnodal}
with \emph{uniform} Zariski non-degeneracy
(see Proposition \ref{prop:GIT:uniformly_non-degen}),
we use the (non-separated, Artin) moduli stack of $r$-marked genus zero curves \emph{up to linear isomorphism}, $\mathcal{C}_{r,1}$.
We control the boundary degeneration of the classifying map
$\phi|_{X\setminus D}$ using geometric invariant theory (GIT)
and semi-stable degenerations.
For the general case of Theorems \ref{thmgenusgreaterone}
and \ref{thmrationalnodal}, however, we use a purely
analytic degeneration argument for marked points
on isotrivial families of curves (Section \ref{sec:one-param_degen}).
This yields results which are beyond what
we are able to prove using the GIT method.

\subsection{Previous work} We contrast our construction with those of earlier works.

\begin{enumerate} \label{previous work}

\item One of our main results (Theorem \ref{thmgenusgreaterone})
implies a result from an (unpublished) thesis of one of the authors \cite[{Theorem 2.2.4}]{rchenthesis}, which gives arithmetic hyperbolicity of duals of smooth plane curves $C$ of non-prime degree $d \geq 6$,
with a prohibition on higher-order tangents lines to $C$,
of order $d$ or $d - 1$
(such as \cite[{Example 2.2.10}]{rchenthesis}).
While the non-prime degree condition can be removed by a variant on loc. cit., the prohibition on high order tangents
seems essential for the degeneration method of loc. cit..
We thus need different degeneration constructions to
prove our main theorems such as Theorem \ref{thmgenusgreaterone},
which allows arbitrary smooth plane curves of degree $d \geq 3$
(arithmetic hyperbolicity fails for the remaining smooth plane curves).

Besides our constructions, we are not aware of
comparably \emph{explicit} examples of geometrically irreducible divisors
$D\subseteq \mathbb{P}^n$ for which one can prove Zariski degeneration or finiteness of sets of integral points (e.g.
Example \ref{ex:fermatcubicdual}).
Previous results of this sort are under the assumption that the divisor $D$ is \emph{general} in a suitable family.

\item Zannier \cite[Sec. 2]{MR2135140} constructs examples of branch loci in $\mathbb{P}^n$ which arise from restricting linear projections $\mathbb{P}^{n+1} \dashrightarrow \mathbb{P}^n$ (equivalently, $\mathbb{P}(\mathcal{O}\oplus \mathcal{O}(1)) \to \mathbb{P}^n)$ to hypersurfaces. Similarly, some of our results can be deduced by studying the $\mathbb{P}^1$-bundle $\mathbb{P}(\Omega_{\mathbb{P}^2}) \to \check{\mathbb{P}}^2$ and applying Theorem \ref{ThmMainCriterion} (see Remark \ref{rem:grasslines}). In Example \ref{ex:projections1}, we show that Theorem \ref{ThmMainCriterion} also recovers Zannier's examples.

\item Faltings \cite{MR1975455}, Zannier \cite[Sec. 3]{MR2135140}, and Levin \cite[Sec. 13]{MR2552103} construct examples in $\mathbb{P}^2$ which arise from a larger class of projections from surfaces $X \hookrightarrow \mathbb{P}^N \dashrightarrow \mathbb{P}^2$. Although our examples also arise as the branch locus of a finite flat cover of the plane, our divisors cannot be realized using their construction (see Remark \ref{rem:branchedcoverdual}).

\item Faltings' family of examples $\mathbb{P}^2 \setminus D$ in \cite{MR1975455} do not admit embeddings into semi-abelian varieties, even after a finite \'etale cover. However, Vojta has reproduced Faltings results by embedding open subsets of such covers into semi-abelian varieties, thereby reducing them to older arithmetic results, see \cite{vojta2009transplantingfaltingsgarden}. Our results eventually reduce
to the arithmetic hyperbolicity of some punctured curves $C$ (Lang--Parry--Mahler--Siegel), the moduli stack of curves of genus $g \geq 2$ (Faltings), the moduli stack of elliptic curves (Shafarevich), and, more generally, the moduli stack of principally polarized abelian varieties (Faltings), see Example \ref{ex:AHcurves}.
The key new input is our degeneration strategy,
sketched in Strategy \ref{strategy},
which we use to lift these classical finiteness results to our settings.
\end{enumerate}




\subsection{Conventions} By a \emph{variety} over any field $k$, we will always mean a geometrically integral separated scheme of finite-type over $k$. If $k$ is a field containing a ring $A$, and if $X$ is a finitely presented $k$-stack, then a \emph{model} of $X$ over $A$ is a finitely-presented $A$-stack $\mathfrak{X}$ with $\mathfrak{X} \times_{A} \Spec k \cong X$. Generally, if $X$ is a $S$-scheme and $T \to S$ is a morphism, we will denote by $X_S$ the base change. The one exception to this is that the base change of a $k$-variety to a $k$-algebra $A$ will be denoted by $X_A$. We sometimes check geometric properties of schemes, stacks, and morphisms (e.g. non-constancy) between them which are all defined over a number field $k$ by passing to the algebraic closure $\bar{k}$. However, we will be clear about which fields we are working over. If $\mathcal{G}$ is a groupoid (for example, the $S$-points of a stack), we will denote by $\pi_0(\mathcal{G})$ the set of isomorphism classes of $\mathcal{G}$. If $V$ is a vector bundle on a scheme $S$, then $\mathbb{P}(V)/S$ parametrizes sub-line-bundles of $V$.


\section{Preliminaries} 

\subsection{Integral points}\label{SecIntegralPoints}

The following definition is central to the article.

\begin{definition} \label{def:integral points} Let $k$ be a number field, $S$ a finite set of primes of $\mathcal{O}_k$, and let $X$ be a proper variety over $k$.  

\begin{enumerate} \item Fix a proper integral model $\mathfrak{X}$ over $\Spec \mathcal{O}_{k,S}$. Given a closed subscheme $D\subseteq X$, \emph{the set $(D,S)$-integral points of $X$ relative to $\mathfrak{X}$} is the set of rational points $x\in X(k)$ such that the closures of $x$ and $D$ on $\mathfrak{X}$ do not meet.

\item We say that a subset $\Sigma\subseteq X(k)$ is \emph{a set $(D,S)$-integral points of $X$} if there is a proper integral model $\mathfrak{X}$ as above such that $\Sigma$ only consists of $(D,S)$-integral points of $X$ relative to $\mathfrak{X}$.

\end{enumerate} \end{definition}

\begin{remark}
In Definition \ref{def:integral points},
note that $D$ is not required to be a divisor.
If $\overline{D} \subseteq \mathfrak{X}$
denotes the closure of $D$, the set of rational points
appearing in Definition \ref{def:integral points}(1) is canonically identified with $(\mathfrak{X} \setminus \overline{D})(\mathcal{O}_{k,S})$.
Note also $X(k)=\mathfrak{X}(\mathcal{O}_{k,S})$ since $\mathfrak{X}$ is proper, so the closure of a rational point is an $\mathcal{O}_{k,S}-$ point.
\end{remark}

\begin{remark} \label{rem:intpointsweilvojta} There is an alternative definition of sets of integral points using Weil functions, which has some indeterminacy up to the choice of error terms \cite{zbMATH03985400}. One can see that these definitions are equivalent in the following sense:

If $\Sigma\subseteq X(k)$ is a set of $(D,S)$-integral points in the sense of Definition \ref{def:integral points}, then there is a choice of Weil functions such that it is a set of $(D,S)$-integral points in the sense of Weil functions. Conversely, if $\Sigma\subseteq X(k)$ is a set of $(D,S)$-integral points in the sense of Weil functions, then there is a possibly larger finite set of primes $T\supseteq S$ and a choice of integral model $\mathfrak{X}$ over $\mathcal{O}_{k,T}$ such that $\Sigma$ consists of $(D,T)$-integral points relative to $\mathfrak{X}$. (One needs to allow a possibly larger $T$ for two reasons: the construction of the model might need spreading-out the generic fiber, and one needs to account for the choice of error term in the Weil functions.) \end{remark}

\begin{definition} \label{def:zardeg} Let $k$ be a number field, and suppose $X$ is a proper variety over $k$.  

\begin{enumerate} 

\item Given a closed subscheme $D \subset X$, we say that \emph{every set of $(D,S)$-integral points is Zariski degenerate} if for every finite set $S$ of primes of $\mathcal{O}_k$, there is a proper closed subscheme $Y \subset X$ containing all sets of $(D,S)$-integral points of $X$.  

\item Given a closed subscheme $D \subset X$, we say that \emph{every set of $(D,S)$-integral points is uniformly Zariski degenerate} if there is a closed subscheme $Y \subset X$ such that for every finite extension $k \subset L$ and every finite set of primes $S$ of $\mathcal{O}_L$, all but finitely many $(D_L,S)$-integral points of $X_L$ lies on $Y_L$. 

\item Given a closed subscheme $D \subset X$, we say that \emph{every set of $(D,S)$-integral points is finite} if for every finite set $S$ of primes of $\mathcal{O}_k$, the set of $(D,S)$-integral points of $X$ is finite.

\end{enumerate}

In situation (1) (resp. (2)) of the above, we
say that $(X,D)$ and $X \setminus D$ have \emph{Zariski degenerate}
(resp. \emph{uniformly Zariski degenerate})
sets of $S$-integral points for all $S$.

\end{definition}

\begin{remark} By Nagata compactification and the above, we may speak of any $k$-variety
$Y$ having Zariski degenerate, uniformly Zariski-degenerate,
or finitely many $S$-integral points. \end{remark}

\begin{remark} \label{rem:onemodelworks} Note that to check that every set of $(D,S)$-points is (uniformly) Zariski degenerate (or finite), it suffices to consider only the sets of $(D,S)$-integral points with respect to one fixed proper integral model $\mathfrak{X}$ over some $\mathcal{O}_{k,S'}$. Indeed, any other such model agrees with $\mathfrak{X}$ after perhaps enlarging $S'$ and so if (uniform) Zariski degeneracy (or finiteness) holds for every set of $(D,S)$-integral points with respect to $\mathfrak{X}$ it does so for any other proper integral model.  \end{remark}

\begin{Example} \label{ex:uniformzdegen} Uniform Zariski degeneracy is stronger than Zariski degeneracy over every number field. Indeed, in the latter, one is allowed to change the closed subset $Y$ with each field extension. For example, consider $\mathbb{P}^1 \times C$ where $C$ is a smooth genus $\geq 2$ curve over $\mathbb{Q}$: over any fixed number field $k$, all the rational points lie on the fibers $\mathbb{P}^1\times p_i$ for the finitely many rational points $p_i \in C(k)$. This is a proper closed subset but one cannot choose it uniformly as the number field $k$ grows. \end{Example}

In order to accommodate moduli spaces, we give a definition of arithmetic hyperbolicity that includes stacks. This definition is due to Javanpeykar-Loughran (see \cite[Def. 4.1]{zbMATH07360996}) but we modify it so that it is a $k$-rational property.

\begin{definition} \label{def:arithhyp} Let $X$ denote a finite-type separated Deligne--Mumford stack over a number field $k$. We say $X$ is \emph{arithmetically hyperbolic} if for any number field $L$ containing $k$, there is a finite set $S$ of primes of $\mathcal{O}_L$ along with a model $\mathfrak{X}$ over $\Spec \mathcal{O}_{L,S}$ of $X_L$ so that the set $\mathrm{Im}[\pi_0(\mathfrak{X}(\mathcal{O}_{L,S'})) \to \pi_0(\mathfrak{X}(L))]$ is finite for all finite sets of places $S'$ containing $S$. \end{definition}

If we restrict to separated Deligne--Mumford
models $\mf{X}$ in Definition \ref{def:arithhyp},
we can simply consider $\pi_0(\mathfrak{X}(\mathcal{O}_{L,S'}))$
instead of the image in $\pi_0(\mathfrak{X}(L))$,
by Proposition \ref{prop:AHequiv}(4) below.

\begin{Example} If $X$ is a proper $k$-variety with a closed subscheme $D \subset X$, then the reader can check that the following are equivalent
\begin{enumerate} \item For any finite extension $L/k$, and any finite set $S$ of primes of $\mathcal{O}_L$, every set of $(D_L,S)$-integral points of $X$ is finite.
\item $X \setminus D$ is arithmetically hyperbolic. 
\end{enumerate} \end{Example}

We compare our definition with that of \cite{zbMATH07360996} in the following proposition (see also \cite[Thm. 4.23]{zbMATH07360996}).

\begin{proposition} \label{prop:AHequiv} If $X$ denotes a finite-type separated Deligne--Mumford stack over a number field $k$, then the following statements are equivalent. 

\begin{enumerate} 

\item $X$ is arithmetically hyperbolic. 

\item $X_{\bar{\mathbb{Q}}}$ is arithmetically hyperbolic in the sense of \cite[Def. 4.1]{zbMATH07360996}: there is a $\mathbb{Z}$-finitely generated subring $A \subset \bar{\mathbb{Q}}$ and a model $\mathfrak{X}$ over $A$ of $X_{\bar{\mathbb{Q}}}$ such that for all $\mathbb{Z}$-finitely generated subrings $A' \subset \bar{\mathbb{Q}}$ containing $A$, the set $\mathrm{Im}[\pi_0(\mathfrak{X}(A')) \to \pi_0(\mathfrak{X}(\bar{\mathbb{Q}}))]$ is finite. 

\item For any $\mathbb{Z}$-finitely generated subring $A \subset \bar{\mathbb{Q}}$ and any model $\mathfrak{J}$ of $X_{\bar{\mathbb{Q}}}$ over $A$ with fraction field $K$ the set $\mathrm{Im}[\pi_0(\mathfrak{J}(A)) \to \pi_0(\mathfrak{J}(K))]$ is finite.

\item For any $\mathbb{Z}$-finitely generated subring $A \subset \bar{\mathbb{Q}}$ which is integrally closed, and all models of $\mathfrak{Y}$ of $X_{\bar{\mathbb{Q}}}$ with finite diagonal, the set $\pi_0(\mathfrak{Y}(A))$ of $A$-integral points is finite. 

\end{enumerate}\end{proposition}

\begin{proof} The last three statements are equivalent by \cite[Thm. 4.23]{zbMATH07360996}. We first show (1) implies (2). Note that there is a finite set of places $S$ of $k$ along with a model $\mathfrak{X}$ of $X$ over $\mathcal{O}_{k,S}$. Clearly $\mathfrak{X}$ is also a model of $X_{\bar{\mathbb{Q}}}$ over $A=\mathcal{O}_{k,S}$, and we will show it satisfies the condition in (2). Now suppose $A' \subset \bar{\mathbb{Q}}$ is a $\mathbb{Z}$-finitely generated subring containing $A$. It must be an order in a number field $L$, so we may assume $A'=\mathcal{O}_{L,S_L}$ for some finite set of places $S_L$ of $L$. We will show the stronger statement $\mathrm{Im}[\pi_0(\mathfrak{X}(\mathcal{O}_{L,S_L})) \to \pi_0(\mathfrak{X}(L))]$ is finite. By definition \ref{def:arithhyp}, after enlarging $S_L$ there is a model $\mathfrak{Y}$ over $\mathcal{O}_{L,S_L}$ of $X_L$ so that $\mathrm{Im}[\pi_0(\mathfrak{Y}(\mathcal{O}_{L,S_L'})) \to \pi_0(\mathfrak{Y}(L))]$ is finite for every $S'_L$ containing $S_L$. But $\mathfrak{Y}$ and $\mathfrak{X}_{\Spec \mathcal{O}_{L,S_L}}$ must agree after enlarging $S_L$, since they are both models of $X_L$, so the desired finiteness follows from that of $\mathfrak{X}$.

We conclude by showing (3) implies (1). Note that given any number field $L$, there is a finite set of places $S$ of $L$ so that there is a model $\mathfrak{X}$ of $X_L$ over $\mathcal{O}_{L,S}$. Now the desired finiteness follows from (3) since $\mathfrak{X}$ is also a model for $X_{\bar{\mathbb{Q}}}$ and the rings $\mathcal{O}_{L,S'}$ are $\mathbb{Z}$-finitely generated. \end{proof}

\begin{Example} \label{ex:AHcurves} There are plenty of interesting examples of arithmetically hyperbolic schemes and stacks. Here we mention just a few milestones in the theory. In fact, the arithmetic input of this paper can be reduced to (1) and (2) below (see Remark \ref{rem:hilbcurveAH} and Lemma  \ref{lem:markedgenus0ah}).
\begin{enumerate} 
\item (Lang--Mahler--Parry--Siegel \cite{zbMATH03181694, zbMATH03007856, zbMATH03059691, Siegel1921}) The punctured projective line $\mathbb{P}^1 \setminus \{0,1,\infty\}$ and a punctured elliptic curve $E \setminus \{p\}$ are both arithmetically hyperbolic.
\item (Faltings, Shafarevich \cite{MR718935, zbMATH03944028}) A smooth geometrically connected projective curve $C$ of genus $g \geq 2$, the stack of smooth projective curves $\mathcal{M}_g$ of genus $g \geq 2$, the stack of elliptic curves $\mathcal{M}_{1,1}$ and, more generally, the stack of principally polarized abelian varieties $\mathcal{A}_g$ of dimension $g \geq 2$ are all arithmetically hyperbolic. 
\item (Faltings \cite{MR1109353}, Vojta \cite{zbMATH00931987, zbMATH01275702}) The complement of an ample divisor on an abelian variety is arithmetically hyperbolic. Vojta has an extension of this result to semi-abelian varieties.

\end{enumerate} 
\end{Example}

\begin{Example} \label{ex:warning} Fix a proper $k$-variety $X$ and fix two closed subsets $D, Y \subset X$. 

\begin{enumerate} 
\item Suppose that every set of $(D,S)$-integral points on $X$ is uniformly Zariski degenerate with respect to $Y$. That is,  all but finitely many of every set of $(D,S)$-integral points must lie on $Y$. This implies that $X \setminus (D \cup Y)$ is arithmetically hyperbolic. 
\item However, if $X \setminus (D \cup Y)$ is arithmetically hyperbolic, then it does not imply that all but finitely many of a set of $(D,S)$-integral points on $X$ must lie on $Y$. An example is given by $D=\emptyset, Y=\{0,1,\infty\} \subset \mathbb{P}^1_k$: there are infinitely many $(D,S)$-points on $\mathbb{P}^1_k$ (so they cannot all lie on $Y$).

\end{enumerate}

\end{Example}

In this paper, a morphism locally of finite type
of schemes is \emph{generically finite} if there exists
a dense open subscheme in the source
which is locally quasi-finite over the target. 

\begin{lemma} \label{lem:nonconstant} Fix a number field $k$, a $k$-variety $Y$, and a finite-type separated Deligne--Mumford stack $\mathscr{M}/k$ which is arithmetically hyperbolic and with coarse moduli space $M$. If there is a morphism $c\colon Y \to \mathscr{M}$ over $k$ such that the induced map $Y \to M$ is quasi-finite (resp. generically finite, non-constant) then every set of $S$-integral points of $Y$ is finite (resp. uniformly Zariski degenerate, Zariski degenerate) for every number field $L / k$.
\end{lemma}

\begin{proof} Consider the morphism $c \colon Y \to \mathcal{M}$, and let $p\colon \mathcal{M} \to M$ denote the coarse moduli space morphism. Suppose $p \circ c$ is quasi-finite (resp. generically finite, non-constant).

Since these objects are finite-type over $k$, we may enlarge $S$ so that $Y,\mathcal{M}, M, c,$ and $p$ spread out to $\mathcal{O}_{k,S}$-flat separated
finite type integral schemes $\tilde{Y}$ and $\tilde{M}$,
an $\mathcal{O}_{k,S}$- finite-type separated algebraic stack with finite diagonal $\tilde{\mathscr{M}}$,
and morphisms $\tilde{Y} \xrightarrow{\tilde{c}} \tilde{\mathscr{M}} \xrightarrow{\tilde{p}} \tilde{M}$.
Our task is to show that the sets of $S$-integral points
$\tilde{Y}(\mathcal{O}_{L,S})$ are finite (resp. uniformly Zariski degenerate, Zariski degenerate) for finite extensions $L / k$.

Since $\mathscr{M}$ is arithmetically hyperbolic, we know that
$\pi_0(\tilde{\mathscr{M}}(\mathcal{O}_{L,S}))$ is finite.
If $g_1,...,g_l$ denote their images in $\tilde{M}(k)$, then $\tilde{Y}(\mathcal{O}_{L,S}) \subseteq Y(k)$ is contained in the fibers of the $g_i$ along the quasi-finite (resp. generically finite, non-constant) map $p \circ c$. In other words, the elements of $\tilde{Y}(\mathcal{O}_{L,S})$ must be finite (resp. uniformly Zariski degenerate, Zariski degenerate) as desired. \end{proof}

\begin{remark} See \cite[Prop. 4.17]{zbMATH07360996} for a precursor to Lemma \ref{lem:nonconstant}. \end{remark}




\section{Moduli spaces and their arithmetic} 

In this section we introduce some important moduli spaces and stacks and investigate their arithmetic properties. First we recall the classical notion of a Hilbert scheme.

\subsection{The Hilbert scheme} Given a flat projective morphism of schemes $X \to S$, the Hilbert functor $\mathrm{Hilb}_{X/S}\colon \mathrm{Sch}/S \to \mathrm{Set}$ sending $T/S \mapsto \{ S$-flat closed subschemes $H \subset X_T$ of finite presentation$\}$ is representable by a projective $S$-scheme. We will use two important examples. 

\begin{Example} \label{ex:hilbcurve} Let $C/k$ be a smooth projective geometrically connected curve over a field $k$. In this case, the open locus of $\mathrm{Hilb}_{C/k}$ consisting of $T$-\'etale closed subschemes $H \subset C_T$ of length $n$ is representable by the scheme-theoretic quotient $(C^n \setminus \Delta^{\mathrm{big}})/S_n$ where $\Delta^{\mathrm{big}}$ consists of those tuples in $C^n$ which do not have distinct entries and the symmetric group $S_n$ acts by permutation. We denote it by $\mathrm{Hilb}^{n, \mathrm{sm}}_{C/k}$. \end{Example}

\begin{Remark} \label{rem:hilbcurveAH} Let $k$ be a number field.
If $C/k$ is a smooth geometrically connected curve of genus $g \geq 2$ and $n \geq 1$ is an integer, then the scheme $\mathrm{Hilb}^{n,\mathrm{sm}}_{C/k}$ is arithmetically hyperbolic. Indeed, $\mathrm{Hilb}^{n,\mathrm{sm}}_{C/k}$ admits a finite \'etale cover by an open subscheme of $C^n$. Thus, by Example \ref{ex:AHcurves}(2) (Faltings) this open subscheme is arithmetically hyperbolic. One can then deduce the claim for $\mathrm{Hilb}^{n,\mathrm{sm}}_{C/k}$ using a result of Chevalley--Weil (see, for instance, \cite[Thm. 1.2]{zbMATH07360996}). \end{Remark}

\begin{Remark} \label{rem:hilbcurve_low_genus_not_AH} Let $k$ be a number field. If $C/k$ is a smooth geometrically connected curve of genus $g \leq 1$ and $n \geq 1$ is an integer, then the scheme $\mathrm{Hilb}^{n,\mathrm{sm}}_{C/k}$ is never arithmetically hyperbolic.
This is due to the large automorphism group of $C$, as we now explain.

In the case of genus $g = 1$,
we can consider disjoint sections $e_1, \ldots, e_n \in C(k)$
(enlarge $k$ if necessary), extend $(C,e_1)$ to an elliptic
curve $\mathcal{C}$ over $\Spec \mathcal{O}_{k,S}$ (enlarge $S$ if necessary),
assume that $e_1, \ldots, e_n \in \mathcal{C}(\mathcal{O}_{k,S})$
are still disjoint sections (enlarge $S$ if necessary),
and assume that $(C,e_1)$ has positive rank (enlarge $k$ if necessary).
Thus $\operatorname{Aut}(\mathcal{C})$ is infinite
(consider translation by a non-torsion point),
but the $S_n$-orbit of $(e_1, \ldots, e_n) \in \mathcal{C}(\mathcal{O}_{k,S})^n$
has finite stabilizer, because
any automorphism of $\mathcal{C}$ fixing $e_1$
must be one of the (finitely many) automorphisms of the elliptic curve
$(\mathcal{C}, e_1)$.
In particular, $\mathrm{Hilb}^{n,\mathrm{sm}}_{\mathcal{C} / \mathcal{O}_{k,S}}(\mathcal{O}_{k,S})$ is infinite.

In the case of genus $g = 0$, we can assume $C \cong \mathbb{P}^1$
(enlarge $k$ if necessary), and abuse notation to also denote $\mathbb{P}^1$
for the usual model over $\Spec \mathcal{O}_{k,S}$.
The case of $n = 1$ is clear, i.e. $\mathbb{P}^1(\mathcal{O}_{k,S}) = \mathbb{P}^1(k)$ is certainly infinite.
The case of $n = 2$ is also clear, i.e. $(\mathbb{P}^1 \setminus \{\infty\})(\mathcal{O}_{k,S}) = \mathbb{A}^1(\mathcal{O}_{k,S})$ is also certainly infinite.
For the case of $n \geq 3$, recall
the canonical identification of group schemes $\operatorname{PGL}_2 \cong \underline{\operatorname{Aut}}(\mathbb{P}^1)$,
and note that $\operatorname{PGL}_2(\mathcal{O}_{k,S})$ is infinite
(consider upper triangular matrices).
But any ordered tuple of disjoint sections $(e_1, \ldots, e_n) \in \mathbb{P}^1(\mathcal{O}_{k,S})$ has trivial stabilizer (if $n \geq 3$) for the (faithful) action by $\operatorname{PGL}_2(\mathcal{O}_{k,S})$.
In particular, $\mathrm{Hilb}^{n,\mathrm{sm}}_{\mathcal{C} / \mathcal{O}_{k,S}}(\mathcal{O}_{k,S})$ is infinite.
\end{Remark}

\begin{Example} \label{ex:hilbproj} Let $\pi\colon \mathbb{P}^n_S \to S$ be projective space and fix an integer $r \geq 1$. The Hilbert subfunctor $\mathrm{Hilb}_{r,n}\colon \mathrm{Sch}/S \to \mathrm{Set}$ sending $T/S \mapsto \{T$-flat closed subschemes $H \subset \mathbb{P}^n_T$ whose fibers are hypersurfaces of degree $r\}$ is representable by $\mathbb{P}(\pi_*\mathcal{O}_{\mathbb{P}^n_S}(r))$, and is the \emph{Hilbert scheme of degree $r$ hypersurfaces in $\mathbb{P}^n_S$}. We will denote by $\mathrm{Hilb}_{r,n}^{\mathrm{sm}} \subset \mathbb{P}(\pi_*\mathcal{O}_{\mathbb{P}^n_S}(r))$ the open locus parametrizing $T$-smooth subschemes. \end{Example}

\subsection{Moduli stacks of curves with a multi-section} 

The moduli stacks $\mathcal{M}_{g,n}$ of genus $g$ curves with $n$ marked (and distinct) points are well studied objects in algebraic geometry. In this section, we will relate these to another stack which will be more useful to us. 

\begin{definition} \label{defmgsigman} Let $g,n \geq 0$ be positive integers. Then we define the \emph{stack of curves of genus $g$ with an \'etale multi-section of degree $n$} to be the category fibered in groupoids $\mathcal{M}_{g,\Sigma_n}$ given by
\[\mathcal{M}_{g,\Sigma_n}(S) = \Set{ (f\colon C \to S, i\colon \Sigma \hookrightarrow C) | \begin{array}{c}
     \: \text{-$f$ is proper and smooth}\\
     \text{-the geometric fibers of $f$ are genus $g$ curves}\\
     \text{-$i$ is a closed immersion}\\
    \text{-The map $\Sigma \to S$ is finite \'etale of degree $n$}\\
  \end{array}}
\]

\noindent Morphisms in $\mathcal{M}_{g,\Sigma_n}(S)$ are $S$-isomorphisms of relative curves which preserve the given multi-sections. 

\end{definition}

\begin{remark} \label{rem:mgsigmastack} It is a standard application of descent theory to show that the category fibered in groupoids $\mathcal{M}_{g,\Sigma_n}$ is a stack.
If $g = 1$, one should allow $C$ to be an algebraic space. \end{remark}

Next, we show $\mathcal{M}_{g,\Sigma_n}$ is algebraic. First, observe that there is a natural functor $F\colon \mathcal{M}_{g,n} \to \mathcal{M}_{g,\Sigma_n}$ which sends $(C/S,\sigma_1,...,\sigma_n)$ to $(C/S, \bigsqcup_{i=1}^n \sigma_i(S) \subset C)$. Second, there is a natural strict action (see \cite[Def. 2.1 (i)]{MR2125542}) of $S_n$ on $\mathcal{M}_{g,n}$ where $\delta \in S_n$ sends $(f\colon C \to S,\sigma_1,...,\sigma_n) \mapsto (f\colon C \to S, \sigma_{\delta(1)},....,\sigma_{\delta(n)})$. 

\begin{proposition} \label{prop:mgsigmaalg} The functor $F$ induces an isomorphism between the quotient stack $[\mathcal{M}_{g,n}/S_n] \cong \mathcal{M}_{g,\Sigma_n}$. In particular, $\mathcal{M}_{g,\Sigma_n}$ is a separated Deligne--Mumford stack when $g \geq 2$ (for all $n \geq 0$), when $g=1$ (for all $n \geq 1$), or when $g=0$ (for all $n \geq 4$). \end{proposition}

 \begin{proof} By \cite[Thm. 4.1]{MR2125542} an $S$-point of the quotient stack $[\mathcal{M}_{g,n}/S_n]$ is the data of an $S_n$-torsor $T \to S$, along with an $S_n$-equivariant map $g: T \to \mathcal{M}_{g,n}$. By the universal property of $\mathcal{M}_{g,n}$, this $g$ is equivalent to a descent datum for an $S$-point of $\mathcal{M}_{g,\Sigma_n}$. Indeed, an $S_n$-equivariant $g$ corresponds to the data of $(f\colon C \to T, \sigma_1,...,\sigma_n)$ and for each $\delta \in S_n$ isomorphisms $\phi_{\delta}\colon C \to C$ fitting into cartesian diagrams:
 \begin{center}
$\begin{CD}
T @>\mathrm{\sigma_i}>> C@>\mathrm{f}>>T\\
@V\mathrm{\delta}VV @V\mathrm{\phi_{\delta}}VV @VV\mathrm{\delta}V\\
T @>\mathrm{\sigma_{\delta(i)}}>> C@>\mathrm{f}>> T
\end{CD}$
 \end{center}

\noindent which are compatible with composition i.e. $\phi_{\delta\delta'}=\phi_{\delta} \circ \phi_{\delta'}$. In other words, the action of $S_n$ on $T$ extends compatibly to actions on $\bigsqcup_{i=1}^r \sigma_i(T)$ and $C$. Thus, by Galois descent (equivalently, taking $S_n$-quotients) we obtain $\Sigma \to C' \to S$, where $C'$ is a genus $g$ curve over $S$ and $\Sigma$ is a closed subscheme of $C'$ which is finite \'etale of degree $n$ over $S$. This defines the functor $[\mathcal{M}_{g,n}/S_n] \to \mathcal{M}_{g,\Sigma_n}$. 
	
	To construct the inverse, note that given $\alpha=(f\colon C \to S, i\colon \Sigma \to C)$ in $\mathcal{M}_{g,\Sigma_n}(S)$, the scheme of $S$-isomorphisms $T=\mathrm{\underline{Isom}}_S(\bigsqcup_{i=1}^r S,\Sigma)$ is an $S_n$-torsor over $S$. Moreover, pulling back $\alpha$ along $t\colon T \to S$ yields a $T$ point of $\mathcal{M}_{g,n}$ with descent datum with respect to the Galois covering $t$. We leave it to the reader to show that these two functors are inverse to each other. 
    
    Lastly, since $F$ is $S_n$-torsor (by \cite[Thm. 4.1]{MR2125542}) it is a finite \'etale cover. Thus, the last statements follow from the corresponding statements about $\mathcal{M}_{g,n}$. This is well known, see, e.g. \cite[Tag 0E9C]{stacks-project}, \cite[Thm. 13.1.2]{Olsson}, and \cite[1.1.2-1.1.5]{MR2262630} for the cases $(g \geq 2, n=0)$, $(g=1, n=1)$, and $(g=0,n=4)$
    (these claims are valid over $\Spec \mathbb{Z}$). For the rest of the cases, note that the map forgetting a point $\mathcal{M}_{g,n} \to \mathcal{M}_{g,n-1}$ realizes $\mathcal{M}_{g,n}$ as the punctured universal curve $\mathcal{C}_{\mathrm{univ}} \setminus \bigsqcup_{i=1}^{n-1} \sigma_i(\mathcal{M}_{g,n-1})$ over $\mathcal{M}_{g,n-1}$.  \end{proof}

\begin{remark} \label{rem:keelmori} Since the stacks $\mathcal{M}_{g,\Sigma_n}$ are separated and Deligne--Mumford, then by \cite[Tag 0DUT]{stacks-project}, they admit coarse moduli spaces. We will denote these by $M_{g,\Sigma_n}$. \end{remark}

We introduce another family of algebraic stacks which will also be useful.

\begin{definition} \label{defhypersurfaces} Let $r \geq 1$ and $n \geq 1$ be positive integers. Then we define the \emph{stack of hypersurfaces of degree $r$ of dimension $n-1$} to be the category fibered in groupoids $\mathcal{C}_{r,n}$ given by
\[\mathcal{C}_{r,n}(S) = \Set{ (f\colon P \to S, i\colon H \hookrightarrow P) | \begin{array}{c}
     \: \text{-$f$ is proper and smooth}\\
     \text{-the geometric fibers of $f$ are isomorphic to $\mathbb{P}^n$}\\
     \text{-$i$ is a closed immersion}\\
     \text{-$H$ is flat and finitely presented over $S$}\\
    \text{-The geometric fibers $H_{\bar{x}}$ are hypersurfaces of degree $r$}\\
  \end{array}}
\]

\noindent Morphisms in $\mathcal{C}_{r,n}(S)$ are $S$-maps of schemes which preserve the closed subschemes $H$. 

\end{definition}

\begin{remark} These stacks are well-known. We enumerate some of their properties below.

 \begin{enumerate}\item The argument given in \cite[Lem. 4.1]{MR4699876}.  shows that there is a natural equivalence of stacks $\mathcal{C}_{r,n} \cong [\mathrm{Hilb}_{r,n}/\mathrm{PGL}_{n+1}]$ (note that although loc.cit only considers smooth objects, the same argument holds in general). As such, the stacks $\mathcal{C}_{r,n}$ are algebraic.

 \item In fact, the moduli stack of smooth hypersurfaces $\mathcal{C}^{\mathrm{sm}}_{r,n}=[\mathrm{Hilb}^{\mathrm{sm}}_{r,n}/\mathrm{PGL}_{n+1}]$ over a field $k$ is a separated Deligne--Mumford stack (see \cite[Thm. 1.6]{MR3022710}). In particular, there exists a coarse moduli space $C^{\mathrm{sm}}_{r,n}$ of smooth degree $r$ hypersurfaces in $\mathbb{P}^n$.

\item  There is a natural identification of $\mathcal{M}_{0,\Sigma_r}$with the open substack of $\mathcal{C}_{r,1}$ which parametrizes the hypersurfaces of $\mathbb{P}^1$ which are smooth of degree $r$.

\end{enumerate}
\end{remark}

Next, we show that the stack $\mathcal{C}_{r,1}$ admits open substacks with good schematic approximations. 

\subsection{GIT quotients} \label{sub:git} Fix $r \geq 3$
and work over a base field $k$. The group scheme $\mathrm{SL}_2$ has a natural linear action on the vector space of degree $r$ binary forms. As such, the $\mathrm{SL}_2$-space $\mathrm{Hilb}_{r,1}=\mathbb{P}(H^0(\mathbb{P}^1,\mathcal{O}(r)))$ admits an $\mathrm{SL}_2$-linearization of $\mathcal{O}_{\mathrm{Hilb}_{r,1}}(1)$ and therefore geometric invariant theory yields the equivariant open immersions $\mathrm{Hilb}_{r,1}^{\mathrm{sm}} \subset \mathrm{Hilb}_{r,1}^{s} \subset \mathrm{Hilb}_{r,1}^{ss} \subset \mathrm{Hilb}_{r,1}$ (equivalently, open immersions $\mathscr{C}_{r,1}^{\mathrm{sm}} \subset \mathscr{C}_{r,1}^s \subset \mathscr{C}_{r,1}^{ss} \subset \mathscr{C}_{r,1}$), parametrizing the smooth, stable and semistable locus, respectively (see, e.g., \cite[Ch. 1 Sec. 4, Prop. 4.2]{GIT}). Also by \cite[pg. 80]{GIT}, the (semi-)stable locus consists of those degree $r$ subschemes of $\mathbb{P}^1$ which only contain points of multiplicity less than (or equal to) $r/2$.
	
	\begin{proposition} \label{prop:GIT} Fix a base field $k$ and let $r \geq 3$, then there is a morphism
	\[\pi\colon \mathscr{C}_{r,1}^{ss} \to C^{ss}_{r,1}=\mathrm{Hilb}_{r,1}^{ss} \sslash \mathrm{SL}_{2}\]
	to a quasi-projective scheme $C^{ss}_{r,1}$ of finite-type. Moreover, there is a nonempty open subscheme $C^{s}_{r,1} \subset C^{ss}_{r,1}$ such that the restricted morphism $\pi^{-1}(C^s_{r,1})=\mathscr{C}_{r,1}^{s} \to C^{s}_{r,1}$ induces a bijection on isomorphism classes of $L$-points whenever $L$ is an algebraically closed field.
\end{proposition}

\begin{proof} By \cite[Thm 1.10]{GIT}, the semistable locus of the Hilbert scheme admits a universal categorical quotient under the action of $\mathrm{SL}_2$, and it is finite-type and quasi-projective. We denote the morphism by $\phi\colon \mathrm{Hilb}^{ss}_{r,1} \to C^{ss}_{r,1}$. This map $\phi$ factors through the stack quotient $\mathrm{Hilb}^{ss}_{r,1} \to [\mathrm{Hilb}^{ss}_{r,1}/\mathrm{PGL}_2] \to C^{ss}_{r,1}$ since it is invariant under $\mathrm{PGL}_2$. Moreover, by \cite[Thm 1.10 (iii)]{GIT} there is an open subscheme $C^{s}_{r,1} \subset C^{ss}_{r,1}$ such that $\phi^{-1}(C^{s}_{r,1})=\mathrm{Hilb}_{r,1}^s$ and $\phi$ restricted to this equivariant open is a geometric quotient. In particular, when restricted to this locus, the fiber over any geometric point is a single orbit. Thus, we obtain a cartesian square 

\[
  \xymatrix{[\mathrm{Hilb}_{r,1}^s/\mathrm{PGL}_2]=\mathscr{C}_{r,1}^s \ar[r] \ar[d] & [\mathrm{Hilb}_{r,1}^{ss}/\mathrm{PGL}_2]=\mathscr{C}_{r,1}^{ss} \ar[d] \\ C^{s}_{r,1} \ar[r] & C^{ss}_{r,1}. }
\] 
Lastly, when $L$ is an algebraically closed field, the isomorphism classes of $[\mathrm{Hilb}_{r,1}^s/\mathrm{PGL}_2](L)$ are precisely the orbits of the action and therefore, the left vertical map induces a bijection on isomorphism classes of $L$ points. \end{proof}

\subsection{Arithmetic hyperbolicity of moduli} Now we discuss the arithmetic hyperbolicity of these moduli spaces.

\begin{lemma} \label{lem:markedgenus0ah} The stack $\mathcal{M}_{g,n,\mathbb{Q}}$ is arithmetically hyperbolic if $g \geq 2$ (for every $n\geq 0$), if $g=1$ (for every $n\geq 1$), or if $g=0$ (for every $n\geq 4$). \end{lemma}

\begin{proof} Recall that $\mathcal{M}_{0,4} \cong \mathbb{P}^1\setminus \{0,1,\infty\}$, and therefore it is arithmetically hyperbolic by Example \ref{ex:AHcurves} (1). When $n=1$ (resp. $n=0$) and $g=1$ (resp. $g \geq 2$) this is a well-known result of Shafarevich (resp. Faltings \cite[Sec. 6, Kor. 1]{MR718935}, see Example \ref{ex:AHcurves}). For the remaining cases, use that the map forgetting the last section $F_n\colon \mathcal{M}_{g,n} \to \mathcal{M}_{g,n-1}$ realizes $\mathcal{M}_{g,n}$ as the punctured universal curve $\mathcal{C}_{\mathrm{univ}} \setminus (\bigsqcup_{i=1}^{n-1} \sigma_{i,\mathrm{univ}}) \to \mathcal{M}_{g,n-1}$. Thus, any $\mathcal{O}_{K,S}$-point of $\mathcal{M}_{g,n-1}$ has finite fibers along $F_n$ by the results of Lang-Mahler-Parry-Siegel (see Example \ref{ex:AHcurves} (1)), and
also Faltings' result (see Example \ref{ex:AHcurves} (2)) for the case $g \geq 2$ and $n = 1$. \end{proof}

\begin{corollary} \label{cor:binaryformsAH} The algebraic stack $\mathcal{M}_{g,\Sigma_n,\mathbb{Q}}$ is arithmetically hyperbolic if $g \geq 2$ (for every $n\geq 0$), if $g=1$ (for every $n\geq 1$), or if $g=0$ (for every $n\geq 4$). \end{corollary}

\begin{proof} By Proposition \ref{prop:mgsigmaalg}, there is a finite \'etale cover $\mathcal{M}_{g,n} \to \mathcal{M}_{g,\Sigma_n}$. By Lemma \ref{lem:markedgenus0ah}, and the stacky Chevalley--Weil theorem (see \cite[Thm. 5.1]{zbMATH07360996}), the result follows. \end{proof}


\section{Duals of smoothly branched curves}
\label{sec:duals}

We will eventually show that the complements of dual varieties of many non-rational curves have few integral points.
This section contains some preliminary material on these dual varieties and we assume throughout that $k$ is a field of characteristic $0$.
First, given a (reduced) closed subscheme $X \subset \mathbb{P}^n$, one may consider the incidence correspondence over the smooth locus $X_{\mathrm{sm}}$:
\[I^0_X=\Set{(p,[H]) | \begin{array}{c}
     \: \text{$p \in H$ and $T_p(X) \subset T_p(H)$}\\
  \end{array}} \subset X_{\mathrm{sm}} \times_{\mathbb{P}^n} \check{\mathbb{P}}^n_k
\]
and take its closure, which we denote by $I_X$, in $X \times_{\mathbb{P}^n} \check{\mathbb{P}}^n_k$.

\begin{definition} \label{def:dual} (see \cite[Ch. 1, Sec. 3a]{zbMATH05248780}) Let $X \subset \mathbb{P}^n$ be a geometrically integral projective $k$-scheme. Then the \emph{dual variety} of $X$,  denoted $X^{*}$, is the scheme-theoretic image of the composition $I_X \subset X \times \check{\mathbb{P}}^n_k \to \check{\mathbb{P}}^n_k$. \end{definition}

\begin{remark} \label{rem:duals} We enumerate some well-known facts about this definition

\begin{enumerate} 

\item The map $I^0_X \to X_{\mathrm{sm}}$ can be identified with the projective bundle $\mathbb{P}(N_{X_{\mathrm{sm}}/\mathbb{P}^n}^{\vee})$, where $N_{X_{\mathrm{sm}}/\mathbb{P}^n}^{\vee}$ is the conormal bundle of $X_{\mathrm{sm}}$ in $\mathbb{P}^n$ (see \cite[p. 27]{zbMATH05248780}).
\item When $X$ is smooth and is not contained in any hyperplane, a hyperplane $[H] \in \check{\mathbb{P}}^n_k$ belongs to $X^{*}$ if and only if the scheme-theoretic intersection $H \cap X$ is not smooth. (see \cite[p. 14]{zbMATH05248780})

\item There is a natural isomorphism $X \cong (X^*)^*$ (see \cite[Thm. 1.1]{zbMATH05248780}).

\item $X^*$ is a geometrically irreducible variety (being the image of the geometrically irreducible scheme $I_X$).

\item If $X$ is a smooth curve of degree $\geq 2$, then $X^*$ is a hypersurface and $I_X \to X^*$ is a resolution of singularities (see \cite[Cor. 1.2 p. 14 and p. 30]{zbMATH05248780}).
\end{enumerate}

\end{remark}

In fact, one can strengthen Remark \ref{rem:duals} (2) when $X$ has controlled singularities.

\begin{definition} \label{def:regim} We say that a geometrically integral curve $\bar{C}$ over a field of characteristic $0$ is \emph{smoothly branched} if the normalization map $C \to \bar{C}$
is unramified.
\end{definition}

The property of being smoothly branched is invariant (and can be checked)
for arbitrary extensions of the base field.

\begin{Remark} \label{ex:nodal} Smoothly branched curves need not be smooth. They are those geometrically integral curves whose completed local rings have a spectrum with regular irreducible components. As an example, recall that a point $p$ on a curve $C$ over field $k$ has a (split) \emph{$r$-fold nodal singularity} if there is an isomorphism 
\[\widehat{\mathcal{O}}_{C,p} \simeq k[[x,y]]/\Pi_{i=1}^r(x+\lambda_iy)\]
where the $\lambda_i \in k$ are distinct. Geometrically integral curves
with all singularities of this type are examples
of smoothly branched curves.
\end{Remark}

It is well-known that if a smooth projective curve $C \subset \mathbb{P}^n$ is not contained in a projective hypersurface $X \subset \mathbb{P}^n$, then $X \cap C$ is singular at $p$ if and only if $T_p(C) \subset T_p(X)$. We extend this to smoothly branched curves below:

\begin{lemma} \label{lemma:dual} Let $f\colon C \to \bar{C} \subset \mathbb{P}^n$ be the normalization of a projective smoothly branched curve $\bar{C}$ which is not contained in any hyperplane, and which is defined over a field of characteristic zero. Then $[H] \in \check{\mathbb{P}}^n_k$ belongs to $(\bar{C})^*$ if and only if $C \times_{\mathbb{P}^n} H$ is singular. \end{lemma}

\begin{proof} Without loss of generality, we may assume $k$ is algebraically closed. Indeed, smoothness is flat-local and the formation of the normalization and dual are compatible with the extension $\bar{k}/k$. 

We show that the closure $I_{\bar{C}}$ of $I^0_{\bar{C}}$ has another description. Consider the incidence $I$ correspondence in $C \times \check{\mathbb{P}}^n$ defined by $\{(p,H)|f^{-1}(H) \textrm{ is singular at } p\}$.
We claim that $\pi \colon I \ra C$ is a $\mathbb{P}^{n-2}$ bundle. Indeed, if $p \mapsto q$ along the normalization of $\bar{C}$, then the fiber of $p$ along $\pi\colon I \to C$ is the projectivization of the kernel of the composition
\[\mathrm{Ker}[H^0(\mathbb{P}^n,\mathcal{O}(1)) \to k(q)] \to \mathfrak{m}_q/\mathfrak{m}_q^2 \to \mathfrak{m}_p/\mathfrak{m}_p^2.\]
Moreover, since $\bar{C}$ is smoothly branched, the second map is surjective.
Put differently, the canonical morphism
$f^* \Omega_{\mathbb{P}^n/k} \rightarrow \Omega_{C/k}$ is surjective
as in \cite[\href{https://stacks.math.columbia.edu/tag/06BB}{Tag 06BB}]{stacks-project},
by the unramifiedness of the normalization map. Thus, the above composition is surjective and the fiber of $\pi$ is isomorphic to $\mathbb{P}^{n-2}$ as desired. Thus, $I$ is irreducible and its (automatically closed) image $f(I)$ in $\mathbb{P}^n \times_k \check{\mathbb{P}}^n$ coincides with $I_{\bar{C}}^0$ over $(\bar{C})_{\mathrm{sm}}$, hence $I_{\bar{C}}=f(I)$. The result now follows. \end{proof}

\begin{remark} \label{rem:branchedcoverdual} Let $\bar{C} \subset \mathbb{P}^2$ be a smoothly branched curve of degree $d \geq 2$ with normalization $C \to \bar{C}$. Then the incidence correspondence $I \subset X=\mathbb{P}(\Omega_{\mathbb{P}^2}|_{C})$ defined in the proof of Lemma \ref{lemma:dual} is isomorphic to $C$ and yields a section of the $\mathbb{P}^1$-bundle $\pi\colon X \to C$. Moreover, the composition 
\[f\colon X \to \mathbb{P}^2 \times \check{\mathbb{P}}^2 \to \check{\mathbb{P}}^2\] 
is a finite flat map of degree $d$ with ramification locus $I$. Thus, the duals of non-linear smoothly branched plane curves can be realized as the branch loci of a finite flat cover of the plane by a smooth ruled surface $X$.

Note that the constructions of Faltings \cite{MR1975455}, Zannier \cite[Sec. 3]{MR2135140}, and Levin \cite[Sec. 13]{MR2552103} also arise as branch loci of ramified coverings of the plane. However, we claim that our branch loci are different. For instance, for any $p \in (\bar{C})_{\mathrm{sm}}$ the map $f$ realizes the fiber $\pi^{-1}(p) \cong \mathbb{P}^1$ as a line inside the dual plane. Thus, it is not possible for $\mathcal{L}=f^*\mathcal{O}(1)$ to generate $3$-jets
(as is required by their construction) i.e. 
\[H^0(X,\mathcal{L}) \to \mathcal{L}_x/\mathfrak{m}^4_x\mathcal{L}_x \]
fails for any $x \in \pi^{-1}(p)$ because $\mathcal{L}|_{\pi^{-1}(p)} \cong \mathcal{O}(1)|_{\mathbb{P}^1}$.

As such, our coverings $X \to \mathbb{P}^2$ are different from those of Faltings, Zannier and Levin. Recall also that their branch loci only have ordinary cusps or nodes. Now, if $\bar{C}^*$ only has ordinary cusps and nodes, then $f\colon X \to \mathbb{P}^2$ constructed above is a \emph{generic cover} as in Chisini's conjecture, in the sense of \cite{zbMATH05496894}. Thus, by Kulikov's work on Chisini's conjecture (see \cite[Thm. 1]{zbMATH05496894}) the dual curves we consider cannot be realized as the branch loci of the generic projections constructed by Faltings, Zannier, and Levin. \end{remark}

\section{Analytic degenerations of marked points
on isotrivial families of curves}
\label{sec:one-param_degen}

In this section we control some degenerations of marked curves
by analytic methods.

\begin{lemma}\label{lemma:invt_Riemannian_metric}
Any connected compact Riemann surface $C$ of genus $g \geq 1$
admits an $\operatorname{Aut}(C)$-invariant Riemannian metric.
\end{lemma}
\begin{proof}
First consider the case of genus $g = 1$.
Pick any point $e \in C$, which equips the pair $(C,e)$
with the (unique) structure of a complex torus (elliptic curve)
with identity point $e$.
Consider any translation-invariant Riemannian metric
$T$ on $(C,e)$. Such $T$ are unique up to positive real scalar.
Recall that $\operatorname{Aut}(C)$ (automorphisms $C$ just as
a compact Riemann surface, not necessarily preserving the
addition law) consists of translations composed
with automorphisms of the elliptic curve $(C,e)$.
Any $\varphi \in \operatorname{Aut}(C,e)$
must preserve $T$ up to positive real scalar, by the uniqueness
from above.
But $\operatorname{Aut}(C,e)$ is a finite group,
so the character $\operatorname{Aut}(C,e) \ra \mathbb{R}^{\times}_{>0}$
given by $\varphi \mapsto (\varphi^* T) / T$ must be trivial.
So $T$ is already $\operatorname{Aut}(C)$-invariant.

The case of $g \geq 2$
follows from finiteness of $\operatorname{Aut}(C)$.
Indeed, if $T$ is any Riemannian metric on $C$,
then the average
$\frac{1}{\operatorname{Aut}(C)} \sum_{h \in \operatorname{Aut}(C)} h^*T$
is an $\operatorname{Aut}(C)$-invariant Riemannian metric on $C$.
\end{proof}

Given a compact Riemann surface $C$ and an understood integer $n \geq 1$,
the \emph{big diagonal} $\Delta^{\mathrm{big}} \subseteq C^n$
consists of the locus of tuples whose coordinate values
are not pairwise distinct.

\begin{corollary}\label{corollary:prox_genus_geq_1}
Let $C$ be a compact Riemann surface of genus $g \geq 1$.
For any integer $n \geq 2$, there exists a continuous proximity function
    \begin{equation}
    \operatorname{prox} \colon C^n \ra \mathbb{R}_{\geq 0}    
    \end{equation}
which is $S_n$-invariant, $\operatorname{Aut}(C)$-invariant,
and whose zero locus is exactly $\Delta^{\mathrm{big}}$.
\end{corollary}
\begin{proof}
Take any invariant Riemannian metric on $C$ as in Lemma \ref{lemma:invt_Riemannian_metric},
and let $d \colon C \times C \ra \mathbb{R}_{\geq 0}$
be the associated (continuous) distance function,
whose zero locus is exactly the diagonal.
We can then take
    \begin{equation}
    \operatorname{prox}(z_1, \ldots, z_n)
        \coloneqq \min_{i \neq j} d(z_i, z_j).
    \end{equation}\end{proof}

\begin{remark} \label{rem:proxgzero} Both Lemma \ref{lemma:invt_Riemannian_metric}
and Corollary \ref{corollary:prox_genus_geq_1}
fail if $C$ instead has genus $0$.
Indeed, pick any identification $C \cong \mathbb{P}^1(\mathbb{C})$
in this case, and consider the faithful action 
$\operatorname{PGL}_2(\mathbb{C}) \acts \mathbb{P}^1(\mathbb{C})$.
The tori $\mathbb{G}_m \hookrightarrow \operatorname{PGL}_2$
preclude the existence of a function as in
Corollary \ref{corollary:prox_genus_geq_1}.
To see this, consider the family of automorphisms $\varphi_t \colon \mathbb{P}^1(\mathbb{C})$ given by $\varphi_t(z) = tz$
for $t \in \mathbb{C}^{\times}$.
Given any $(z_1, \ldots, z_n) \in C^n \setminus \Delta^{\mathrm{big}}$
with all $z_i \neq \infty$,
we have $(\varphi_t(z_1), \ldots, \varphi_t(z_n))$
converging to $(0, \ldots, 0)$ as $t \rightarrow 0$,
but any $\mathrm{Aut}(C)$-invariant proximity function as in Corollary \ref{corollary:prox_genus_geq_1}
would have to be nonzero at $(0, \ldots, 0)$, which is a contradiction. \end{remark}

\begin{definition}
    \label{def:crossratio}
    
Let $C$ be a compact Riemann surface of genus zero, and
consider the classical holomorphic cross ratio
    \begin{equation}
    \operatorname{CR} \colon C^4 \setminus \Delta^{\mathrm{big}} \rightarrow
    \mathbb{C} \setminus \{ 0, 1 \}
    \end{equation}
(this is the unique $\Aut(C)$-invariant
function with the property that $\opn{CR}(\infty,0,1,z) = z$).
As in \cite[{Definition 2.2.32}]{rchenthesis}
we define the \emph{proximity} function
    \begin{equation}\label{equation:prox_P^1}
    \operatorname{prox} \colon C^n \setminus \Delta^{\mathrm{big}} \ra \mathbb{R}_{>0}
    \quad\quad
    \opn{prox}(z_1, \ldots, z_n) \coloneqq
        \min_{\substack{i,j,k,l \\ \text{pairwise distinct}}}
        |\operatorname{CR}(z_i, z_j, z_k, z_l)|
    \end{equation}
for $n \geq 4$. \end{definition} 

\begin{remark} Note that $\opn{prox}$ is continuous, $S_n$-invariant
and $\Aut(C)$-invariant. \end{remark}
The proximity function defined in Definition \ref{def:crossratio} was introduced in \cite{rchenthesis}
to study degenerations of branch points for cyclic
covers of $\mathbb{P}^1$.

\begin{lemma}\label{lemma:prox_limit_P^1}
Let $C$ be a connected compact Riemann surface
of genus $g = 0$.
Given an integer $n \geq 4$,
suppose $z_1(t), \ldots, z_n(t)$
are continuous functions $[0,1] \ra C$
with the property
    \begin{equation}
    (z_1(t), \ldots, z_n(t)) \in
    \begin{cases}
    C^n \setminus \Delta^{\mathrm{big}} & \text{for $t \in (0,1]$}
    \\
    \Delta^{\mathrm{big}} & \text{for $t = 0$.}
    \end{cases}
    \end{equation}
Assume moreover that the multiset
$\{z_1(t), \ldots, z_n(t) \}$ at $t = 0$
does not contain any element of multiplicity $\geq n - 1$.
Then we have $\lim_{t \rightarrow 0} \opn{prox}(z_1(t), \ldots, z_n(t)) = 0$.
\end{lemma}
\begin{proof}
This is \cite[{Lemma 2.2.34}]{rchenthesis}
(with a typo in loc. cit., where $t \rightarrow 0$ should read
$t \rightarrow 1$).
The condition on the multiset $\{z_1(t), \ldots, z_n(t) \}$
is sharp.
\end{proof}

For the purpose of Section \ref{sec:one-param_degen},
we consider the nonstandard notation
$\pi_0(\mathscr{M}_{g,\Sigma_d}(\mathbb{C})) |_{C}$
which denotes those elements of $\pi_0(\mathscr{M}_{g,\Sigma_d}(\mathbb{C}))$
whose underlying genus $g$ curve is isomorphic to $C$.
For fixed $C$ of any genus,
any $\opn{prox}$ as in Corollary \ref{corollary:prox_genus_geq_1}
or \eqref{equation:prox_P^1}
descends to a map of sets
    \begin{equation}\label{equation:prox_on_moduli}
    \opn{prox} \colon \pi_0(\mathscr{M}_{g,\Sigma_d}(\mathbb{C})) |_{C} \rightarrow \mathbb{R}_{>0}
    \end{equation}
for $d \geq 2$ if $g \geq 1$, and $d \geq 4$ if $g = 0$.

\begin{lemma}\label{lemma:topological_path_lifting}
Let $f \colon X \rightarrow Y$
be a closed map of Hausdorff topological spaces.
If $f$ has finite fibers, then it has the following path lifting property:
every commutative solid diagram of topological spaces
    \begin{equation}
    \begin{tikzcd}
    {(0,1]} \arrow{r}{\tilde{\gamma}} \arrow{d} & X \arrow{d}{f}
    \\
    {[0,1]} \arrow[dotted]{ur}{s} \arrow{r}{\gamma} & Y
    \end{tikzcd}
    \end{equation}
admits a unique dotted arrow making the diagram commute.
\end{lemma}
\begin{proof}

The uniqueness of the lift, if it exists, follows because $X$ is Hausdorff. As such, we prove existence. Consider the finitely many pre-images $x_1, \ldots, x_d \in X$
of $\gamma(0)$. Select disjoint neighborhoods
$U_1, \ldots, U_d$ of $x_1, \ldots, x_d$ and denote by $V_i$ the pre-images $ \tilde{\gamma}^{-1}(U_i)$. Observe that $\bigsqcup_{i=1}^d V_i$ contains $(0,\epsilon)$ for some $\epsilon>0$. If not, then since $f$ is closed there is a sequence $(t_j)_{j \geq 1}$ in $(0,1]$ approaching $0$ for which $(\gamma(t_j))_{j \geq 1}$ does not approach $\gamma(0)$, a contradiction. This implies the finite disjoint union $\bigsqcup_{i=1}^d V_i$ covers the interval $(0,\epsilon)$. Since the interval is connected, there is exactly one $V_i$ containing it. Setting $s(0)=x_i$ gives the desired lift.\end{proof}

\begin{remark}
The previous lemma fails if $f$ is only assumed to be proper,
e.g. the projection $S^1 \times [0,1] \rightarrow [0,1]$ fails the given path lifting property
(consider the section over $(0,1]$ given by $t \mapsto e^{2 \pi i / t}$).
\end{remark}

\begin{lemma}\label{lemma:analytic_path_lifting}
Let $f \colon I \ra Y$ be a finite flat degree $d \geq 1$
morphism of complex analytic spaces.
Assume also that $f$ is generically \'etale.
Let $\gamma \colon [0,1] \rightarrow Y$ be any continuous path
whose restriction $\gamma|_{(0,1]}$ factors through the \'etale locus of $f$ in $Y$.

Then there exist continuous lifts
$s_1, \ldots, s_d \colon [0,1] \rightarrow I$
of $\gamma$, unique up to permutation, characterized
by the equality
    \begin{equation}\label{equation:analytic_path_lifting}
    \{s_1(t), \ldots, s_d(t)\} = I_{\gamma(t)}(\mathbb{C})
    \end{equation}
of multi-sets for all $t \in [0,1]$.
\end{lemma}
\begin{proof}
By \emph{generically \'etale}
in the lemma statement,
we mean that there exists an open subset $U \subseteq Y$
with nowhere dense closed analytic complement,
such that $f$ is \'etale over $U$.
By the \emph{\'etale locus of $f$ in $Y$},
we mean the largest open set $U$ for which this is true.
In the lemma statement, the notation $I_{\gamma(t)}(\mathbb{C})$
is abuse of notation for the multi-set (of size $d$) associated
to the discrete complex analytic space $I_{\gamma(t)}$
(fiber of $I \rightarrow Y$ over $\gamma(t)$),
where we count each $\mathbb{C}$-point of $I_{\gamma(t)}$
with multiplicity given by the length of its local ring.

If $U \subseteq Y$ denotes the \'etale locus of $f$,
we know that $f$ is finite \'etale over $U$,
and so it is a finite degree $d$ covering map
of underlying topological spaces
(e.g. \cite[{Theorem, \S 2.3.2}]{GR84}
and \cite[{Proposition 1.9}]{SHC6061n13}).
By covering space theory, we obtain
unique lifts $s_1, \ldots, s_d$ of $\gamma$ over $(0,1]$,
characterized by the formula \eqref{equation:analytic_path_lifting}
(the $s_i(t)$ are distinct for any fixed $t \in (0,1]$).

By Lemma \ref{lemma:topological_path_lifting},
each $s_i$ extends uniquely to a lift of $\gamma$
on all of $[0,1]$
(this question is local on $Y$ near $\gamma(0)$,
so we may assume that $Y$ is Hausdorff,
which forces $I$ to be Hausdorff because $f$ is separated).
The only remaining task is to show the equality
of multi-sets in \eqref{equation:analytic_path_lifting}
at the point $t = 0$.
By \cite[{Theorem, \S 2.3.2}]{GR84}, we may find a
Hausdorff neighborhood $W$ of $\gamma(0)$ with the property
that $f^{-1}(W) = U_1 \sqcup \cdots \sqcup U_r$
where each $U_j$ is an open neighborhood of
$z_j \in I$, where $\{z_1, \ldots, z_r\}$
is the set of (distinct) points in the fiber $I_{\gamma(0)}$,
counted without multiplicity.
Each composite $U_j \ra W$ is finite flat,
hence open by \cite[{Theorem 2.12, Chapter V}]{BS76}.

For each $z_j$, let $d_{z_j} \in \mathbb{Z}_{> 0}$ be its multiplicity,
i.e. the length of the associated local ring in the fiber $I_{\gamma(0)}$.
By cohomology and base change as in
\cite[{Corollary 3.9, Chapter III}]{BS76},
we conclude that all $y \in Y$ in a sufficiently small
neighborhood of $\gamma(0)$ have fiber $U_j \times_{Y,y} \Spec \mathbb{C}$
of total length $d_{z_j}$ (i.e. sum the lengths of local rings).
We thus find $\# \{i : s_i(0) = z_j\} \leq d_{z_j}$
for all $j$, which implies the claim.\end{proof}
\begin{remark}\label{remark:degenerating_path_existence} Consider $f \colon I \rightarrow Y$ as in the statement
of Lemma \ref{lemma:analytic_path_lifting}.
We call any path $\gamma$ as in the lemma
statement a \emph{degenerating path} for $\gamma(0) \in Y$.
We claim that every $y \in Y$ admits a degenerating path.
If $y$ lies in the \'etale locus on $Y$, this is trivial.
Otherwise, $y$ lies in the closed nowhere dense analytic subset
$Z \subseteq Y$ over which $I \to Y$ is not \'etale,
and we would like a continuous path
which intersects $Z$ only at $y$.
To see that such a degenerating path
$\gamma$ always exists (for any closed nowhere dense analytic $Z \subseteq Y$),
one can work locally on $Y$,
take a resolution of singularities (Hironaka) to reduce to the case
where $Y$ is a complex ball in some $\mathbb{C}^n$,
and then note that $y$ lying in
any proper closed analytic subspace $Z \subseteq \mathbb{C}^n$
admits such a degenerating path, upon connecting
$y$ with any $y' \in \mathbb{C}^n \setminus Z$ by a complex line $L$
and then noting that any proper closed analytic subspace of $L$ is discrete. \end{remark}
\begin{remark} If $f \colon I \ra Y$ is a
finite flat degree $d$ morphism
of schemes which are locally of finite type over $\mathbb{C}$,
with e.g. $Y$ moreover irreducible and $f$ generically \'etale
(meaning \'etale over a dense open subset of $Y$),
then the analytification $f^{\mathrm{an}} \colon I^{\mathrm{an}} \rightarrow Y^{\mathrm{an}}$ satisfies the conditions of
Lemma \ref{lemma:analytic_path_lifting}.
Preservation of the relevant properties of $f$
under analytification is explained in \cite{SGA1ExpXII}.
For any $y \in Y(\mathbb{C})$, the scheme-theoretic fiber $I_y$
and the complex-analytic-space-theoretic fiber $I^{\mathrm{an}}_y$
coincide (as locally ringed spaces),
because complex analytification preserves fiber products,
and any finite scheme over $\Spec \mathbb{C}$ is its own analytification.

We are mostly interested in this algebro-geometric case
for the present paper,
but we find the complex analytic formulation in
Lemma \ref{lemma:analytic_path_lifting} convenient for
the proof, where we would like to lift paths to different
analytic branches near ramification points
of the source. \end{remark}

For the next lemma, we consider the following situation.

\begin{setup} 
Consider a commutative diagram of complex analytic spaces
    \begin{equation}
    \begin{tikzcd}
    I \arrow{dr}[swap]{f} \arrow{r} & X \arrow{d}{\pi}
    \\
    & Y
    \end{tikzcd}
    \end{equation}
where $Y$ is reduced,
and where $\pi$ is a proper smooth morphism which
is moreover isotrivial, with fibers all isomorphic to a
fixed connected compact Riemann surface $C$ of genus $g$. Fix an integer $d \geq 1$, and suppose $I \ra X$ is a closed immersion
such that $f \colon I \rightarrow Y$ is finite flat of degree $d$ and generically \'etale with \'etale locus $U \subset Y$. We refer to $Y \setminus U$ as the branch locus of $f$.
\end{setup}

\begin{remark} \label{rem:classsets} In the setup above, for each $y \in U$, the data $(X_y, I_y)$
defines an object of $\pi_0(\mathscr{M}_{g,\Sigma_d}(\mathbb{C}))$.
We thus obtain a \emph{classifying map} of sets
$c\colon U(\mathbb{C}) \ra \pi_0(\mathscr{M}_{g,\Sigma_d}(\mathbb{C}))$.
\end{remark}

\begin{lemma}\label{lemma:analytic_degen_classifying_map}
Take the setup above.
    \begin{enumerate}
        \item Assume $g \geq 1$ and $d \geq 2$. If the branch
            locus of $I \rightarrow Y$ is nonempty,
            then the classifying map $c\colon U(\mathbb{C}) \ra \pi_0(\mathscr{M}_{g,\Sigma_d}(\mathbb{C}))$
            is non-constant.
        \item Assume $g = 0$ and $d \geq 4$. If the branch
            locus of $I \rightarrow Y$ contains a point $y \in Y$
            with the property that the fiber $I_y$ contains
            no non-reduced point of multiplicity $\geq d - 1$,
            then the classifying map $c\colon U(\mathbb{C}) \ra \pi_0(\mathscr{M}_{g,\Sigma_d}(\mathbb{C}))$
            is non-constant. 
    \end{enumerate}
\end{lemma}
\begin{proof} If $g \geq 1$, let $y \in Y$ be any point in the branch locus.
If $g = 0$, let $y \in Y$ be a point in the branch locus
satisfying the hypotheses of the lemma statement.

The isotrivial fibration $\pi \colon X \rightarrow Y$
is in fact a trivial product family, locally on $Y$,
by a theorem of Fischer--Grauert \cite{FG65} as generalized by
Poljakov \cite{Poljakov69},
so we may assume $X = Y \times C$ without loss of generality.
See also \cite[{\S 3}]{SHC6061n13} for 
equivalent characterizations of smooth morphisms of complex analytic spaces,
or ``morphismes simples'' in the language of loc. cit..

Take any degenerating path $\gamma \colon [0,1] \rightarrow Y$ for $y$,
in the sense of Remark \ref{remark:degenerating_path_existence}.
Take liftings $s_1, \ldots, s_d$ of $\gamma$
as in Lemma \ref{lemma:analytic_path_lifting}.
Take any continuous proximity function
as in Corollary \ref{corollary:prox_genus_geq_1} (genus $g \geq 1$)
or \eqref{equation:prox_P^1} (genus $g=0$).
If the classifying map $c\colon U(\mathbb{C}) \ra \pi_0(\mathscr{M}_{g,\Sigma_d}(\mathbb{C}))$ were constant after pullback along $\gamma|_{(0,1]}$,
then \eqref{equation:prox_on_moduli}
implies that the continuous function
    \begin{equation}\label{equation:analytic_degen_classifying_map:path_prox}
    \begin{tikzcd}[row sep = tiny]
    (0,1] \arrow{r} & \mathbb{R}_{>0}
    \\
    t \arrow[mapsto]{r} & \opn{prox}(s_1(t), \ldots, s_d(t))
    \end{tikzcd}
    \end{equation}
must be constant.
But in all cases, we
have $\lim_{t \rightarrow 0} \opn{prox}(s_1(t), \ldots, s_d(t)) = 0$
by Corollary \ref{corollary:prox_genus_geq_1} (genus $\geq 1$)
and Lemma \ref{lemma:prox_limit_P^1} (genus $g = 0$),
so \eqref{equation:analytic_degen_classifying_map:path_prox}
cannot be constant.
\end{proof}

\begin{remark}\label{remark:analytic_degen_classifying_map}
Consider the setup of Lemma \ref{lemma:analytic_degen_classifying_map}
but with $I, X, Y$ instead schemes locally of finite type over $\mathbb{C}$,
with $Y$ irreducible, with the adjectives ``proper'', ``smooth'', ``isotrivial'', ``finite flat of degree $d$'',
and ``generically \'etale'' given their usual scheme-theoretic meanings
(with ``generically \'etale'' taken to mean ``\'etale over a dense open
of the target'').
Then the statement Lemma \ref{lemma:analytic_degen_classifying_map}
holds verbatim. This is immediately implied by the version
for complex analytic spaces (we again refer the reader to \cite{SGA1ExpXII}
for comparison of properties of morphisms of schemes locally of finite
type over $\mathbb{C}$ with properties of their analytifications).
\end{remark}

\section{Arithmetic hyperbolicity via non-rational curves}

\begin{Theorem} \label{thm:deg4} Fix a number field $k$ and suppose $\bar{C} \subset \mathbb{P}^n$ is a projective smoothly branched curve which is not contained in any hyperplane. Let $C \to \bar{C}$ be its normalization and suppose $C$ has genus $\geq 1$. Then $\check{\mathbb{P}}^n \setminus (\bar{C})^*$ is arithmetically hyperbolic. \end{Theorem}

\begin{proof} Let $d$ denote the degree of $\bar{C}$.
Note $d \geq 2$. We deal with the case when the genus of $C$ is $g \geq 2$ first. There is a natural morphism of functors $f\colon \check{\mathbb{P}}^n_k \to \mathrm{Hilb}^{d}_{C/k}$ sending a $T$-flat family of hyperplanes $\mathcal{H} \subset \mathbb{P}^n_T$ to the pullback $C_T \times_{\mathbb{P}^n_T} \mathcal{H} \subset C_T$. This is a $T$-finite-flat family of subschemes in $C$ by \cite[Tag 00MF]{stacks-project} and by B\'ezout's theorem it is of degree $d$ over $T$. Moreover, by Lemma \ref{lemma:dual}, $f^{-1}(\mathrm{Hilb}^{d,\mathrm{sm}}_{C/k})=\check{\mathbb{P}}^n_k \setminus (\bar{C})^*$. Now since $\mathrm{Hilb}^{d,\mathrm{sm}}_{C/k}$ is arithmetically hyperbolic (see Remark \ref{rem:hilbcurveAH}), it suffices to show $f$ is quasi-finite over $\check{\mathbb{P}}^n \setminus (\bar{C})^*$. This is immediate: a family of hyperplanes $\mathcal{H}$ parametrized by a closed curve $Y \subset \check{\mathbb{P}}^n$ which is not contained in $(\bar{C})^*$ must intersect $(\bar{C})^*$ (because $(\bar{C})^*$ is a hypersurface) and therefore $ C_Y \times_{\mathbb{P}^n_Y} \mathcal{H}  \subset C_Y$ supports \'etale and non-\'etale closed subschemes over $Y$.

For the genus $g=1$ case, let $\mathcal{H} \subset \mathbb{P}^n \times \check{\mathbb{P}}^n$ denote the universal hyperplane and set $I=\mathcal{H}|_{C \times \check{\mathbb{P}}^n} \subset C \times \check{\mathbb{P}}^n$. The fiber of the projection $I \subset C \times \check{\mathbb{P}}^n \to \check{\mathbb{P}}^n$ over a hyperplane $[H]$ is $H \times_{\mathbb{P}^n} C \subset C$, and by Lemma \ref{lemma:dual} this is an \'etale divisor of degree $d$ in $C$ when $[H]$ doesn't lie on $(\bar{C})^*$. Thus, $I$ is finite \'etale of degree $d$ over $\mathbb{P}^n \setminus (\bar{C})^*$, hence there are classifying maps $F\colon \check{\mathbb{P}}^n \setminus (\bar{C})^* \to \mathcal{M}_{1,\Sigma_d} \to M_{1,\Sigma_d}$. By Corollary \ref{cor:binaryformsAH} and Lemma \ref{lem:nonconstant}, it suffices to show $F$ doesn't contract any curve $Y^{\circ} \subset \check{\mathbb{P}}^n \setminus (\bar{C})^*$. The non-constancy of $F|_{Y^{\circ}}$ can be checked after base change along an embedding $k \to \mathbb{C}$, so we now assume $k=\mathbb{C}$. Let $Y$ be the closure of $Y^{\circ}$. Then there is a point $p \in Y \cap (\bar{C})^*$ and, perhaps after replacing $Y$ with an open neighborhood of $p$ in $Y$, we see that Lemma \ref{lemma:analytic_degen_classifying_map}
(and Remark \ref{remark:analytic_degen_classifying_map}) imply that $F|_{Y^{\circ}}$ cannot be constant. \end{proof}

\begin{remark}
The proof method of Theorem \ref{thm:deg4}
for the case of genus $g = 1$ in fact works for the case
of genus $g \geq 2$ as well.
It will also work for the case $C$ of genus $g = 0$.
We postpone the genus $g = 0$ case to Proposition \ref{lemma:dualrational},
as the statement needs more hypotheses to be correct
(see Example \ref{ex:sharpness}). The genus $g \geq 1$
is already enough for some applications, as we explain next.
\end{remark}

\begin{corollary} \label{cor:deg4plane} Let $k$ be a number field.
For any smooth projective geometrically integral plane curve $C$ over $k$,
the variety $\check{\mathbb{P}}^2 \setminus C^*$ is arithmetically hyperbolic
if and only if $\deg(C) \geq 3$.\end{corollary}
\begin{proof}
The case of $\deg(C) \geq 3$ is immediate from Theorem \ref{thm:deg4}.
If $\deg(C) = 2$, then $\deg(C^*) = 2$ and $\check{\mathbb{P}}^2 \setminus C^*$
cannot be arithmetically hyperbolic (because
it contains copies of $\mathbb{P}^1 \setminus \{ 0, \infty \}$,
which is not arithmetically hyperbolic).
If $\deg(C) = 1$, then $C$ is a line and its dual is simply a point.
\end{proof}

As an application, we answer a question of Achenjang--Morrow, see \cite[Q.1]{MR4749147}: if $k$ is a number field and $X$ is a geometrically connected, smooth, and projective $k$-scheme, then there is a geometrically irreducible divisor $D \subset X$ such that $X \setminus D$ is arithmetically hyperbolic. In fact, our method shows that the degree of $D$ can be controlled. For integers $r \geq 0$, recall that a locally Noetherian scheme $X$ is said to be $S_r$ if every local ring $\mathcal{O}_{X,p}$ has $\mathrm{depth}(\mathcal{O}_{X,p}) \geq \mathrm{min}\{r,\mathrm{dim}(\mathcal{O}_{X,p})\}$.
If $X$ is locally of finite type over a field, then having property $S_r$ is preserved after extending the base field (see, e.g., \cite[Thm. 23.3]{matsumura}). We begin with a lemma.

\begin{lemma} \label{lem:kleiman} Fix a field $k$ of characteristic zero, and suppose $\Delta \subset \mathbb{P}^n$ is a geometrically integral ample divisor. If $X \to \mathbb{P}^n$ is a finite morphism, $X$ has dimension $\geq 2$ and is a geometrically integral $S_2$ projective scheme, then there is an open subscheme $U \subset \mathrm{PGL}_{N,k}$ with $U(k)\neq \varnothing$ such that for any rational point $g \in U(k)$, the fiber product $g\Delta \times_{\mathbb{P}^n} X$ is a geometrically integral divisor of $X$. \end{lemma}

\begin{proof} Apply the conclusion of \cite[Thm. 2 (i), Rem. 7, Thm. 2 (ii)]{MR360616} (with the adjective "locally integral" in place of "regular") to the maps $X_{\bar{k}} \to \mathbb{P}^N_{\bar{k}} \leftarrow \Delta_{\bar{k}}$ to obtain a dense open $U \subset \mathrm{PGL}_{N,\bar{k}}$ with the property that for any $\bar{g} \in U(\bar{k})$ the product $\bar{D}=\bar{g}\Delta_{\bar{k}} \times_{\mathbb{P}^N_{\bar{k}}} X_{\bar{k}}$ is a locally integral divisor of $X_{\bar{k}}$. Since $X_{\bar{k}}$ is $S_2$ and integral, $\bar{D}$ must also be connected by \cite[Tag 0FD9]{stacks-project}, so it must be integral. A standard Galois descent argument shows that there is a nonempty open subscheme $V \subset \mathrm{PGL}_{N,k}$ with $V_{\bar{k}} \subset U$. Moreover, since $\mathrm{PGL}_{N,k}$ is rational, $V$ has a $k$-rational point and if $g$ is in $V(k)$ the fiber product $g\Delta \times_{\mathbb{P}^n} X$ is geometrically integral. \end{proof} 

\begin{corollary} \label{cor:morrowachen} Let $k$ be a number field and suppose $X$ is a geometrically integral projective $k$-scheme which is $S_2$ and of dimension $d \geq 2$. Fix a very ample divisor $H$. Then for every even integer $s\geq  2d+2$, there is a geometrically irreducible divisor $D \subset X$ such that $X \setminus D$ is arithmetically hyperbolic and $D \in |sH|$. \end{corollary}

\begin{proof} Let $C/k$ be a fixed smooth projective geometrically integral curve of genus $1$. Suppose $s$ is an even integer $\geq 2d+2$ and let $L$ denote a line bundle on $C$ of degree $\frac{s}{2}$. Since $C$ has genus $1$ and $L$ has degree $\geq 3$, $L$ is necessarily very ample. Moreover, $h^0(C,L) = \frac{s}{2} \geq d+1$, so the complete linear system of $L$ yields an embedding $C \to \mathbb{P}^{n}_k$ where $n \geq d$. Moreover, by Theorem \ref{thm:deg4}, $\mathbb{P}^n \setminus C^*$ is arithmetically hyperbolic and by \cite[Ex. 7.7]{tevelev2001projectivelydualvarieties} $C^*$ is a hypersurface of degree $2\mathrm{deg}(L)=s$. 

Now consider the embedding defined by $H$, $f\colon X \to \mathbb{P}^n$ and by projecting along an appropriate linear subspace we obtain a finite morphism $X \to \mathbb{P}^d_k$, which we embed as a linear subspace of $\mathbb{P}^n$:
\begin{center}
\begin{tikzcd}
X \arrow[r, hook] \arrow[dr]
& \mathbb{P}^n \arrow[d, dashed]\\
& \mathbb{P}^d_k \arrow[r, hook, "\mathrm{linear}"] & \mathbb{P}^n
\end{tikzcd}
\end{center}

Now apply Lemma \ref{lem:kleiman} to the morphisms $X \to \mathbb{P}^n$ and $C^* \to \mathbb{P}^n$, to obtain a $g \in \mathrm{PGL}_{N,k}(k)$ and a translate $gC^*$ so that $X \times_{\mathbb{P}^n} C^*=D$ is a geometrically integral divisor in $X$. Since $gC^*$ is linearly equivalent to $C^*$, it follows that $D \in |sH|$, as desired. The arithmetic hyperbolicity follows because $X \setminus D \to \mathbb{P}^n \setminus gC^*$ is quasi-finite and the latter is arithmetically hyperbolic.\end{proof}


\section{A construction for Zariski degeneracy and arithmetic hyperbolicity
using marked curve fibrations} 

In this section, we establish the following general construction for Zariski degeneracy and arithmetic hyperbolicity,
using families of marked curves.

\begin{Theorem}\label{ThmMainCriterion}
Let $k$ be a number field, and suppose $X$ is a geometrically integral separated scheme of finite type over $k$. Let $\pi\colon P\to X$ be a proper and smooth morphism of schemes whose geometric fibers are smooth projective connected curves of a fixed genus $g$.
Fix an integer $r \geq 2$
if $g \geq 1$,
and an integer $r \geq 4$ if $g = 0$.

Let $I \subseteq P$ and $\Delta\subsetneq X$ be closed subschemes such that above $X \setminus \Delta$, the map $\pi|_I$ is finite \'etale of degree $r$. Consider the following properties:
\begin{itemize}
\item[(i)] There exists a geometric point $x\in \Delta(\bar{k})$ above which $I$ is finite and flat (after base-change to $\overline{k}$), and the scheme-theoretic fiber $I_x \to \Spec \bar{k}$ is non-reduced. If $g = 0$, we assume all points in $I_x$ have multiplicity $\leq r-2$.
\item[(ii)] The map $\pi|_I$ is finite flat and for every geometric point $x\in \Delta(\bar{k})$, the scheme-theoretic fiber $I_x \to \Spec \bar{k}$ is non-reduced. If $g = 0$, we assume all points in $I_x$ have multiplicity $\leq r-2$. Suppose also that $\Delta$ meets all closed curves in $X$ (e.g. if $X$ is proper and $\Delta$ is given by an effective strictly nef Cartier divisor). \end{itemize}
If (i) (resp. (ii)) holds then every set of $S$-integral points
on $X \setminus \Delta$ is Zariski degenerate (resp. finite).
\end{Theorem}

\begin{proof} Since $I \to P \to X$ is finite \'etale of degree $r$ away from $\Delta$, we get a classifying map $F\colon X \setminus \Delta \to \mathcal{M}_{g,\Sigma_r} \to M_{g,\Sigma_r}$. By Lemma \ref{lem:nonconstant}, to prove the result it suffices to show that $F$ is non-constant, if $(i)$ holds (resp. quasi-finite, if $(ii)$ holds). 

Suppose some (resp. every) fiber of $I$ over $\Delta$
is non-reduced, with the additional prohibition on points of multiplicity
$\geq r-1$ in the case of $g = 0$. Then we may take some (resp. any) geometrically integral $Y \subseteq X$ (enlarge $k$ if necessary) passing through both $X \setminus \Delta$ and $\Delta$,
in such a way that at least one fiber of $I$ over $Y \cap \Delta$ is a non-reduced scheme, such that the geometric fiber moreover
only contains points of multiplicity $\leq r-2$
if $g = 0$. Let $p \in Y \cap \Delta$ be one such branch point. It now suffices to show that the classifying map $F$ restricted to $Y$ is a non-constant morphism. We may base change along an embedding $k \to \mathbb{C}$ to check this, so we replace $k$ with $\mathbb{C}$.
If $P|_Y$ is not isotrivial, the classifying map is
certainly non-constant. If $P|_Y$ is isotrivial
(i.e. all geometric fibers are isomorphic to a fixed curve),
then Lemma \ref{lemma:analytic_degen_classifying_map}
(and Remark \ref{remark:analytic_degen_classifying_map}) 
imply that
that $F$ cannot be constant.
\end{proof}

In the rest of this section, we consider applications
of Theorem \ref{ThmMainCriterion} in the case of genus $g = 0$.
As a straightforward application, consider the following example. 

\begin{Example} \label{ex:projections1} (Projections) Let $X \subset \mathbb{P}^{n+1}$ be a hypersurface of degree $d \geq 3$ for $n \geq 1$. Fixing a point $p$ not on $X$, we project $\mathbb{P}^{n+1} \setminus \{p\} \dashrightarrow \mathbb{P}^{n}$ (see, e.g. \cite[Tag 0B1N]{stacks-project}). The induced map $f\colon X \to \mathbb{P}^{n}$ is finite flat of degree $r$,  with some branch divisor $D \subset \mathbb{P}^{n}$. Blowing up gives a projective bundle $\pi\colon \mathbb{P}(\mathcal{O}_{\mathbb{P}^n} \oplus \mathcal{O}_{\mathbb{P}^n}(1)) \to \mathbb{P}^n$ with a section, $E$. The multi-section defined by $\sigma \colon Y=E \sqcup X \to \mathbb{P}^n$ is of degree $\geq 4$ over $\mathbb{P}^n$. If $|f^{-1}(x)| \geq 2$ (equivalently, if $|\sigma^{-1}(x)| \geq 3$) for some (resp. all) $x \in D$, then Theorem \ref{ThmMainCriterion} implies $\mathbb{P}^n \setminus D$ has Zariski degeneracy of integral points (resp. is arithmetically hyperbolic).  \end{Example}

\begin{remark} \label{rem:zansec2} The result of Example \ref{ex:projections1} is originally due to Zannier, see \cite[Sec. 2]{MR2135140}. \end{remark}

\begin{Example} \label{ex:grassmannians} (Grassmannians of lines) For $n \geq 2$, let $X \subset \mathbb{P}^n$ denote a smooth projective hypersurface of degree $d \geq 4$ which does not contain a line, and suppose $G=\mathrm{Gr}(2,n+1)$ denotes the Grassmannian of lines in $\mathbb{P}^n$. Then there is an associated geometrically irreducible divisor $\Delta_X \subset G$ whose complement has Zariski degenerate integral points. Indeed, the universal line $P=\cL \subset \mathbb{P}^n_k \times G$ is a $\mathbb{P}^1$-bundle over $G$ and, by B\'ezout's theorem, the restriction $Y=\cL \times_{\mathbb{P}^n} X \subset \cL$ is finite flat of degree $d$ over $G$. Moreover, its discriminant locus $\Delta_X$ is the image of the incidence correspondence 
\[J_X=\Set{(p,[L]) | \begin{array}{c}
     \: \text{$p$ is a singular point of $L \cap X$}\\
  \end{array}} \subset \cL \times_{\mathbb{P}^n} X.
\]
via the projection to $G$. In fact, the fibers of $J_X \to X$ are isomorphic to $\mathrm{Gr}(1,n-1)$ so $J_X$ is geometrically irreducible of dimension $2n-3$. It follows that $\Delta_X$ is a geometrically irreducible divisor of $G$, as claimed. We may conclude by showing that the data $Y \subset P$ induces a non-constant morphism $G \setminus \Delta_X \to \mathcal{M}_{0,\Sigma_d} \to M_{0,\Sigma_d}$. Indeed, one may apply Theorem \ref{ThmMainCriterion}(i) by noting that the general line $[L]$ in $\Delta_X(\overline{k})$ intersects $X$ simply, that is:
\[L \cap X=\Spec \overline{k}[\epsilon]/(\epsilon^2) \sqcup \Spec \overline{k} \sqcup \dots \sqcup \Spec \overline{k}.\]
\end{Example}

\begin{remark} \label{rem:grasslines} Specializing to the case $n=2$ in Example \ref{ex:grassmannians} yields the Zariski degeneracy of integral points of some complements of dual plane curves. However, as Corollary \ref{cor:deg4plane} and Proposition \ref{lemma:dualrational} below show, this method is not optimal. Indeed, the moduli maps given by moving a marking along a fixed curve gives more information than that of moving a marked line as in Example \ref{ex:grassmannians}. \end{remark}
 
\begin{proposition} \label{lemma:dualrational} Let $k$ be a number field and suppose $\bar{C} \subset \mathbb{P}^n$ is a smoothly branched rational curve with normalization $f\colon C \to \bar{C} \subset \mathbb{P}^n$. If $\bar{C}$ is not contained in any hyperplane and $\mathrm{deg}(\bar{C})=d \geq 4$, then every set of $S$-integral points of $\mathbb{P}^n_k \setminus \bar{C}^*$ is Zariski degenerate. If, moreover, every hyperplane $H$ has the property that $C \times_{\mathbb{P}^n} H$ has no point of multiplicity $\geq d-1$, then $\check{\mathbb{P}}^n_k \setminus \bar{C}^*$ is arithmetically hyperbolic. \end{proposition}

\begin{proof} Let $f\colon C \to \bar{C} \subset \mathbb{P}^n$ denote the normalization of $\bar{C}$. Now consider the universal hyperplane $\mathcal{H} \subset  \mathbb{P}^n \times_k \check{\mathbb{P}}^n_k$ and its pullback along  $C \times_k \check{\mathbb{P}}^n_k=C_{\check{\mathbb{P}}^n_k}$. This yields a closed subscheme $Y$ of a $\mathbb{P}^1$-bundle $P$:
\[Y=C_{\check{\mathbb{P}}^n_k} \times_{\mathbb{P}^n \times_k \check{\mathbb{P}}^n_k} \mathcal{H} \subset C_{\check{\mathbb{P}}^n_k}=P\]
which is flat over $\check{\mathbb{P}}^n_k$ by \cite[Tag 00MF]{stacks-project}. Moreover, by B\'ezout's theorem it is finite of degree $d$ and even \'etale over $\check{\mathbb{P}}^n_k \setminus (\bar{C})^*$ by Lemma \ref{lemma:dual}. Thus by Theorem \ref{ThmMainCriterion} it suffices to show that for at least one fiber over $(\bar{C})^*$, the geometric fiber $Y_{\bar{x}} \subset \mathbb{P}^1_{\bar{x}}$ doesn't have a point of multiplicity $d-1$. 

However, this follows because the general point on $(\bar{C})^*$ represents a hyperplane $H'$ in $\mathbb{P}^n$ which is tangent to exactly one point $q$ of $\bar{C}$ and with intersection multiplicity exactly $2$ at $q$. To see this observe that  $\bar{C}$ has finitely many inflection points (see, e.g., \cite[Thm. 7.13]{MR3617981} or \cite[p. 37, C-2]{MR770932}). Therefore, the locus $B$ in $I_{\bar{C}}$ consisting of pairs $(p,[H])$ where $H \cap C$ has intersection multiplicity $\geq r-1$ is a proper closed subscheme. As such, it does not surject onto $(\bar{C})^*$ and by the birationality of the surjection $I_{\bar{C}} \to (\bar{C})^*$ we may choose a hyperplane $[H'] \in (\bar{C})^*$ with the desired property. 

The statement about arithmetic hyperbolicity is similar: we apply Theorem \ref{ThmMainCriterion} (ii). \end{proof}

The following example shows that Proposition \ref{lemma:dualrational} is sharp in the following sense: in contrast to the case of nonzero genus (see Theorem \ref{thm:deg4}), if $\bar{C}$ as in Proposition \ref{lemma:dualrational} has high order tangents then $\check{\mathbb{P}}^n_k \setminus (\bar{C})^*$ need not be arithmetically hyperbolic. 

\begin{Example} \label{ex:sharpness}
Fix an integer $d \geq 2$.
Let $f(t) \in \overline{\mathbb{Q}}[t]$
be any polynomial of degree $d - 1$ with distinct roots such that
$f(0) = 0$.
Let $\bar{C}$ be the degree $d$ smoothly branched rational curve in $\mathbb{P}^2$
given by the (closure of the) parameterization
    \begin{equation}
    t \mapsto [x(t) : y(t) : z(t) ] \quad\quad
    x(t) \coloneqq t^d
    \quad
    y(t) \coloneqq f(t)
    \quad
    z(t) = 1.
    \end{equation}
This map has everywhere nonvanishing derivative
(and the same holds at $t = \infty$).
Note that the projective lines $x = 0$ and $z = 0$
are tangent to $\bar{C}$ with maximum multiplicity $d$,
at the points $(0,0,1)$ and $(1,0,0)$ respectively,
corresponding to $t = 0$ and $t = \infty$ respectively.

The associated dual curve $(\bar{C})^*$ is given by the
    (closure of the) rational parameterization    \begin{equation}\label{equation:example:dual_curve_parameterization}
    t \mapsto [ -y'(t) : x'(t) : x(t) y'(t) - y(t) x'(t)]
    =
    [-f'(t) : d t^{d - 1} : t^d f'(t) - f(t) d t^{d - 1}].
    \end{equation}
Note that $(\bar{C})^*$ has degree $2(d - 1)$.
In the coordinates $[a : b : c]$ on $\check{\mathbb{P}}^2$,
the line $L$ given by $b = 0$
intersects $(\bar{C})^*$ only at
the images of $t = 0$ and $t = \infty$, i.e. $[1 : 0 : 0]$ and
$[0 : 0 : 1]$ (corresponding to the two high order tangents to
$\bar{C}$ from above).

In other words, we have
$L \cap (\check{\mathbb{P}}^2 \setminus (\bar{C})^*) \cong \mathbb{P}^1 \setminus \{ 0 , \infty\}$. This is not arithmetically hyperbolic,
so $\check{\mathbb{P}}^2 \setminus (\bar{C})^*$ also cannot
be arithmetically hyperbolic (after descent to any number field).
\end{Example}

On the other hand, the Zariski degeneracy in Proposition \ref{lemma:dualrational} is usually uniform when $n=2$. Indeed, recall that the proof of Proposition \ref{lemma:dualrational} constructs a $\mathbb{P}^1$-bundle $\pi\colon P \to \check{\mathbb{P}}^2_k$ along with a closed subscheme $Y \subset P$ (finite flat of degree $d$ over $\check{\mathbb{P}}^2_k$). Here the fiber of $Y \subset P$ over a point $[L]$ is $L \times_{\mathbb{P}^2} C \subset C \cong \mathbb{P}^1$ and therefore $Y$ is \'etale precisely over $\check{\mathbb{P}}^2_k \setminus (\bar{C})^*$ by Proposition \ref{lemma:dual}. This data yields a morphism to the moduli stack of hypersurfaces in $\mathbb{P}^1$:
\[c\colon \check{\mathbb{P}}^2_k \to \mathcal{C}_{d,1}\]
and we may restrict $c$ to the semi-stable locus $c^{-1}(\mathcal{C}_{d,1}^{\mathrm{ss}})=U^{\mathrm{ss}}$ (see \ref{sub:git}). We call a point of $x \in \check{\mathbb{P}}^2_k$ \emph{semi-stable} if it lies on $U^{\mathrm{ss}}$. We prove:

\begin{proposition} \label{prop:GIT:uniformly_non-degen} Let $k$ be a number field and suppose $\bar{C} \subset \mathbb{P}^2_k$ is a smoothly branched rational curve. If $\mathrm{deg}(\bar{C})=d \geq 5$ and $(\bar{C})^*$ contains a semi-stable closed point $x$ such that $Y_{\bar{x}}$ is not isomorphic to the scheme 
\[\Spec \overline{k}[\epsilon]/(\epsilon^2) \sqcup \Spec \overline{k} \sqcup \dots \sqcup \Spec \overline{k},\]
(this is satisfied if, for example, $(\bar{C})^*_{\overline{k}}$ contains a simple node). Then every set of $S$-integral points of $\check{\mathbb{P}}^2_k \setminus (\bar{C})^*$ is uniformly Zariski degenerate. \end{proposition}

\begin{proof} By Theorem \ref{lem:nonconstant}, it suffices to show that the morphism
\[p \circ c\colon U \to C^{\mathrm{ss}}_{d,1}\]
is generically finite (or equivalently, that the image of $(p \circ c)(U)$, $W$, is not a curve). If we can show this, then $\check{\mathbb{P}}^2_k \setminus (\bar{C})^* \to M_{0,\Sigma_d} \subset C^{s}_{r,n}$ is generically finite and the result follows.

Suppose $(\bar{C})^*(\overline{k}) \cap U^{\mathrm{ss}}$ contains a line $[L_0]$ such that $L_0 \times_{\mathbb{P}^2} C$ is not isomorphic to 
\[\Spec \overline{k}[\epsilon]/(\epsilon^2) \sqcup \Spec \overline{k} \sqcup \dots \sqcup \Spec \overline{k}.\]
However, the general point of $(\bar{C})^*$ corresponds to a line $L$ with 
\[L \times_{\mathbb{P}^2} C \cong \Spec \overline{k}[\epsilon]/(\epsilon^2) \sqcup \Spec \overline{k} \sqcup \dots \sqcup \Spec \overline{k},\]
 and is therefore a \emph{stable} point of $(\bar{C})^*$ (since $d \geq 5$). It follows that the image of $(\bar{C})^* \cap U^{\mathrm{ss}}$ under $p \circ c$ consists of at least two distinct points and by the irreducibility of $(\bar{C})^*$, this image, $W'$, must be positive dimensional. On the other hand, if $W$ was a curve then it must intersect $C^{\mathrm{ss}}_{d,1} \setminus C^{\mathrm{sm}}_{d,1}$ at a zero dimensional scheme. This is a contradiction since $W' \subset W \cap C^{\mathrm{ss}}_{d,1} \setminus C^{\mathrm{sm}}_{d,1}$. \end{proof}



\section{Acknowledgments}

The authors would like to thank Ariyan Javanpeykar, Amos Turchet,
and Shou-Wu Zhang for their comments and suggestions.
This research was partially conducted during the period
R.C. served as a Clay Research Fellow.
N.G.-F. was supported by ANID Fondecyt Regular grant 1251300 from Chile. S.M. was supported by ANID Fondecyt Regular grant 1230402 from Chile. H.P. was supported by ANID Fondecyt Regular grant 1230507 from Chile.

\bibliography{mybib}{}
\bibliographystyle{plain}

\end{document}